\documentclass[10pt, reqno]{amsart}

\usepackage[usenames,dvipsnames]{pstricks}
\usepackage{epsfig}

\usepackage{color}
\date{\today}

\usepackage[latin1]{inputenc}
\usepackage[T1]{fontenc}
\usepackage{amsmath, amsthm}
\usepackage[english]{babel}
\usepackage{amssymb}
\usepackage{latexsym}
\usepackage{graphicx}
\usepackage[all]{xy}

\newtheorem{apart}{}[section]
\newtheorem{teo}[apart]{{\bf Theorem}}
\newtheorem{prop}[apart]{{\bf Proposition}}
\newtheorem{lem}[apart]{{\bf Lemma}}
\newtheorem{ej}[apart]{{\bf Example}}
\newtheorem{ejs}[apart]{{\bf Examples}}
\newtheorem{rem}[apart]{{\bf Remark}}

\newtheorem{defin}[apart]{{\bf Definition}}

\newcommand{\ot}{\otimes}
\newcommand{\co}{\circ}

\begin{document}

\begin{center}
 {\huge{\bf The group of strong Galois objects associated 

 \vspace{0.2cm}
to a cocommutative Hopf quasigroup}}
\end{center}

\ \\

\begin{center}
{\bf J.N. Alonso \'Alvarez$^{1}$, J.M. Fern\'andez Vilaboa$^{2}$ and
R. Gonz\'{a}lez Rodr\'{\i}guez$^{3}$}

\end{center}

\ \\
\hspace{-0,25cm}$^{1}$ Departamento de Matem\'{a}ticas, Universidad
de Vigo, Campus Universitario Lagoas-Marcosende, E-36280 Vigo, Spain
(e-mail: jnalonso@uvigo.es)
\ \\
\hspace{-0,25cm}$^{2}$ Departamento de \'Alxebra, Universidad de
Santiago de Compostela,  E-15771 Santiago de Compostela, Spain
(e-mail: josemanuel.fernandez@usc.es)
\ \\
\hspace{-0,25cm}$^{3}$ Departamento de Matem\'{a}tica Aplicada II,
Universidad de Vigo, Campus Universitario Lagoas-Mar\-co\-sen\-de, E-36310
Vigo, Spain (e-mail: rgon@dma.uvigo.es)
\ \\

\begin{center}
{\bf Abstract}
\end{center}
{\small $\hspace{0.45cm}$   Let $H$ be a cocommutative  faithfully flat Hopf quasigroup in a strict symmetric 
monoidal category with equalizers. In this paper we introduce the notion of (strong) Galois $H$-object and we prove that the set of isomorphism classes  of (strong) Galois $H$-objects is a (group) monoid  which coincides, in the Hopf algebra setting, 
with the Galois group of $H$-Galois objects introduced by Chase and Sweedler.}

{\small }

\vspace{0.5cm}



\section*{Introduction}  Let $R$ be a commutative ring with unit. The notion of Galois $H$-object for a commutative, cocommutative Hopf $R$-algebra $H$, which is a finitely generated projective $R$-module, is due to Chase and Sweedler \cite{Chase-Sweedler}. As was pointed by Beattie \cite{Beattie}, although the discussion of Galois $H$-objects  in \cite{Chase-Sweedler} is limited to commutative algebras, the main properties  can be easily extended to non commutative algebras.  One of more relevant is the following: if $H$ is cocommutative,  the isomorphism classes of Galois $H$-objects form a group denoted by $Gal(R,H)$.  The product  in $Gal(R,H)$ is defined by the kernel of a suitable morphism and the class of $H$ is the identity element. This construction  can be extended to symmetric closed categories with equalizers and coequalizers working with monoids instead of algebras and some of the more important properties and exact sequences involving the group  $Gal(R,H)$ were obtained in this  categorical setting (\cite{Pura}, \cite{Manel}, \cite{Galois}). 

An  interesting 
generalization of Hopf algebras  are Hopf quasigroups 
introduced by Klim and Majid in \cite{Majidesfera} in order to
understand the structure and relevant properties of the algebraic
$7$-sphere. They are not
associative but the lack of this property is compensated by some
axioms involving the antipode. The concept of Hopf quasigroup is a
particular instance of the notion of unital coassociative
$H$-bialgebra introduced in \cite{PI2} and includes the example
of an enveloping algebra $U(L)$ of a Malcev algebra (see
\cite{Majidesfera}) as well as the notion of quasigroup algebra $RL$
of an I.P. loop $L$. Then, quasigroups unify I.P. loops and Malcev
algebras in the same way that Hopf algebras unified groups and Lie
algebras. 

In this paper we are interested to  answer the following question:  is it possible to extend the construction of $Gal(R,H)$  to the situation where $H$ is a cocommutative Hopf quasigroup? in other words,  can we construct in a non-associative setting  a group of Galois $H$-objects? The main obstacle to define the group is the lack of associativity because we must  work with 
unital magmas, i.e. objects where there exists a non-associative product with unit. As we can see in the first section of this paper,  Hopf quasi groups are examples of these algebraic structures.  

The paper is organized as follows. We begin  introducing the notion of right $H$-comodule magma, where $H$ is a Hopf quasigroup, and defining the product  of right $H$-comodule magmas. In the second section we introduce the notions of Galois $H$-object and strong Galois $H$-objects proving  that, with the product defined in the first section for comodule magmas, the set of isomorphism classes forms a monoid, in the case of  Galois $H$-objects,  and a group when we work with strong Galois $H$-objects. In this point it appears the main difference between our Galois $H$-objects  and the ones associated to a Hopf algebra because in the Hopf algebra setting the inverse of  the class of a Galois $H$-object $A$ is the class of the opposite Galois $H$-object $A^{op}$, while in the quasigroup context this property fails. We only have the following: the product of $A$ and $A^{op}$ is isomorphic to $H$ only as comodules. To obtain an isomorphism of magmas we need to work with strong Galois $H$-objects. Then, the strong condition appears in a natural way and we want to point out that in the classical case of Galois $H$-objects associated to a Hopf algebra $H$ all of them are strong.  Finally, in the last section, we study the connections between Galois $H$-objects and invertible comodules with geometric normal basis.

 Throughout this paper $\mathcal C$ denotes a
strict  symmetric monoidal category with equalizers where $\ot$
denotes the tensor product, $K$  the unit object and $c$ the
symmetry isomorphism. We denote the class of objects of ${\mathcal
C}$ by $\vert {\mathcal C} \vert $ and for each object $M\in \vert
{\mathcal C}\vert$, the identity morphism by $id_{M}:M\rightarrow
M$. For simplicity of notation, given objects $M$, $N$ and $P$ in
${\mathcal C}$ and a morphism $f:M\rightarrow N$, we write $P\otimes
f$ for $id_{P}\otimes f$ and $f \otimes P$ for $f\otimes id_{P}$. We
will say that $A \in \vert {\mathcal C}\vert$ is flat if the functor
$A\ot -:{\mathcal C}\rightarrow {\mathcal C}$ preserves equalizers.
If  moreover $A\ot -$  reflects isomorphisms we say
that $A$ is faithfully flat.

By a unital  magma in ${\mathcal C}$ we understand a triple
$A=(A, \eta_{A}, \mu_{A})$ where $A$ is an object in ${\mathcal C}$
and $\eta_{A}:K\rightarrow A$ (unit), $\mu_{A}:A\otimes A
\rightarrow A$ (product) are morphisms in ${\mathcal C}$ such that
$\mu_{A}\circ (A\otimes \eta_{A})=id_{A}=\mu_{A}\circ
(\eta_{A}\otimes A)$. If $\mu_{A}$ is associative, that is,
$\mu_{A}\circ (A\otimes \mu_{A})=\mu_{A}\circ (\mu_{A}\otimes A)$,
the unital magma will be called a monoid in ${\mathcal C}$.  For any unital
magma $A$ with $\overline{A}$ we will denote the opposite unital
magma $(A, \eta_{\overline{A}}=\eta_{A},
\mu_{\overline{A}}=\mu_{A}\co c_{A,A})$. Given two unital magmas
(monoids) $A= (A, \eta_{A}, \mu_{A})$ and $B=(B, \eta_{B},
\mu_{B})$, $f:A\rightarrow B$ is a morphism of unital magmas (monoids) 
if $\mu_{B}\circ (f\otimes f)=f\circ \mu_{A}$ and $
f\circ \eta_{A}= \eta_{B}$. By duality, a counital comagma in
${\mathcal C}$ is a triple ${D} = (D, \varepsilon_{D},
\delta_{D})$ where $D$ is an object in ${\mathcal C}$ and
$\varepsilon_{D}: D\rightarrow K$ (counit), $\delta_{D}:D\rightarrow
D\otimes D$ (coproduct) are morphisms in ${\mathcal C}$ such that
$(\varepsilon_{D}\otimes D)\circ \delta_{D}= id_{D}=(D\otimes
\varepsilon_{D})\circ \delta_{D}$. If $\delta_{D}$ is coassociative,
that is, $(\delta_{D}\otimes D)\circ \delta_{D}=
 (D\otimes \delta_{D})\circ \delta_{D}$, the counital comagma will be called a comonoid.
 If ${D} = (D, \varepsilon_{D},
 \delta_{D})$ and
${ E} = (E, \varepsilon_{E}, \delta_{E})$ are counital comagmas
(comonoids), $f:D\rightarrow E$ is  morphism of counital magmas
(comonoids) if $(f\otimes f)\circ \delta_{D}
=\delta_{E}\circ f$ and  $\varepsilon_{E}\circ f =\varepsilon_{D}.$

Finally note that if  $A$, $B$ are unital magmas (monoids) in
${\mathcal C}$, the object $A\otimes B$ is a unital  magma (monoid)
in ${\mathcal C}$ where $\eta_{A\otimes B}=\eta_{A}\otimes \eta_{B}$
and $\mu_{A\otimes B}=(\mu_{A}\otimes \mu_{B})\circ (A\otimes
c_{B,A}\otimes B).$ With $A^{e}$ we will denote the unital magma
$\overline{A}\ot A$. In a dual way, if $D$, $E$ are counital
comagmas (comonoids) in ${\mathcal C}$, $D\otimes E$ is a  counital
comagma (comonoid) in ${\mathcal C}$ where $\varepsilon_{D\otimes
E}=\varepsilon_{D}\otimes \varepsilon_{E}$ and $\delta_{D\otimes
E}=(D\otimes c_{D,E}\otimes E)\circ( \delta_{D}\otimes \delta_{E}).$

\section{ Comodule magmas for Hopf quasigroups}

This first section is devoted to the study of the notion  of
$H$-comodule magma associated to a Hopf quasigroup $H$. We will show
that, as in the Hopf algebra setting, it is possible to define a
product using suitable equalizers which induces a
monoidal structure in the category of flat $H$-comodule magmas.

The notion of Hopf quasigroup was introduced in \cite{Majidesfera}
and the following is its monoidal version.
\begin{defin}
\label{Hopf quasigroup} {\rm  A Hopf quasigroup $H$   in ${\mathcal
C}$ is a unital magma $(H, \eta_H, \mu_H)$ and a comonoid $(H,
\varepsilon_H, \delta_H)$ such that the following axioms hold:
\begin{itemize}
\item[(a1)] $\varepsilon_H$ and $\delta_H$ are  morphisms of unital magmas.

\item[(a2)] There exists  $\lambda_{H}:H\rightarrow H$
in ${\mathcal C}$ (called the antipode of $H$) such that:

\begin{itemize}
\item[(a2-1)] $\mu_H\circ (\lambda_H\ot \mu_H)\circ (\delta_H\ot H)=
\varepsilon_H\ot H= \mu_H\circ (H\ot \mu_H)\circ (H\ot \lambda_H\ot
H)\circ (\delta_H\ot H).$

\item[(a2-2)] $\mu_H\circ (\mu_H\ot H)\circ (H\ot \lambda_H\ot H)\circ (H\ot \delta_H)=
H\ot \varepsilon_H= \mu_H\circ(\mu_H\ot \lambda_H)\circ (H\ot
\delta_H).$
\end{itemize}
\end{itemize}

If $H$ is a Hopf quasigroup, the antipode  is unique,
antimultiplicative, anticomultiplicative and leaves the unit and the
counit invariable:
\begin{equation}
\label{anti} \lambda_{H}\circ \mu_{H}=\mu_{H}\circ
(\lambda_{H}\otimes \lambda_{H})\circ
c_{H,H},\;\;\;\;\delta_{H}\circ \lambda_{H}=c_{H,H}\circ
(\lambda_{H}\otimes \lambda_{H})\circ \delta_{H},
\end{equation}
\begin{equation}
\label{fixed} \lambda_{H}\circ
\eta_{H}=\eta_{H},\;\;\;\;\varepsilon_{H}\circ
\lambda_{H}=\varepsilon_{H}.
\end{equation}
(\cite{Majidesfera}, Proposition 4.2 and \cite{EmilioPura},
Proposition 1). Note that by (a2),
\begin{equation}
\label{antipode}\mu_{H}\co (\lambda_{H}\ot id_{H})\co \delta_{H}=\mu_{H}\co (id_{H}\otimes \lambda_{H})\co \delta_{H}= \varepsilon_{H}\otimes \eta_{H}.
\end{equation}

A Hopf quasigroup $H$ is cocommutative if $c_{H,H}\co
\delta_{H}=\delta_{H}$. In this case, as in the Hopf algebra
setting, we have that $\lambda_{H}\co \lambda_{H}=id_{H}$ (see
Proposition 4.3 of \cite{Majidesfera}).

Let  $H$ and $B$ be Hopf quasigroups. We say that $f:H\rightarrow B$
is a morphism of Hopf quasigroups if it is a morphism of unital
magmas and  comonoids. In this case $\lambda_B\circ f=f\circ
\lambda_H$ (see Proposition 1.5 of \cite{NikaRamon5}). }
\end{defin}

\begin{ejs}
{\rm The notion of Hopf quasigroup was introduced in \cite{Majidesfera}
and it can be interpreted as the linearization of the concept of
quasigroup. A quasigroup is a set $Q$ together with a product such
that for any two elements $u, v \in Q$ the equations $u  x = v$, $x
u = v$ and $u v = x$ have  unique solutions in $Q$. A quasigroup $L$
which contains an element $e_{L}$ such that $ue_{L} = u = e_{L}  u $
for every $u \in L$ is called a loop. A loop $L$ is said to be a
loop with the inverse property  (for brevity an I.P. loop) if and
only if, to every element $u\in L$, there corresponds an element
$u^{-1}\in L$ such that the equations $u^{-1}(uv)=v=(vu)u^{-1}$ hold
for  every $v\in L$.

If $L$ is an I.P. loop, it is easy to show (see \cite{Bruck}) that
for all $u\in L$ the element $u^{-1}$ is unique and
$u^{-1}u=e_{L}=uu^{-1}$. Moreover, for all $u$, $v\in L$, the
equality $(uv)^{-1}=v^{-1}u^{-1}$ holds.

Let $R$ be a commutative ring and $L$ and I.P. loop. Then, by
Proposition 4.7 of \cite{Majidesfera}, we know that
$$RL=\bigoplus_{u\in L}Ru$$
is a cocommutative Hopf quasigroup with  product given by the
linear extension of the one defined in $L$ and
$$\delta_{RL}(u)=u\ot u, \;\; \varepsilon_{RL}(u)=1_{R}, \;\;
\lambda_{RL}(u)=u^{-1}$$ on the basis elements.

 Now we briefly describe another example of Hopf quasigroup constructed working with Malcev algebras (see \cite{PIS} for details). Consider a commutative and associative ring  $K$ with $\frac{1}{2}$ and $\frac{1}{3}$ in $K$. A Malcev algebra  $(M,[\;,\;])$ over $K$ is a free module in  $K-Mod$ with a bilinear anticommutative
operation [\;,\;] on $M$ satisfying that $[J(a, b, c), a] = J(a, b, [a, c]),$  where $J(a, b, c) = [[a, b], c] - [[a, c], b] - [a, [b, c]]$ is the Jacobian in $a, b, c$. Denote by $U(M)$ the not necessarily associative algebra  defined as the quotient of $K\{M\}$, the free non-associative algebra on a basis of $M$, by the ideal $I(M)$ generated by the set $\{ab-ba-[a,b], (a,x,y)+(x,a,y), (x,a,y)+(x,y,a) :   a, b\in M, x,y \in K\{M\} \},$
where $(x,y,z)=(xy)z-x(yz)$ is the usual additive associator.

By  Proposition 4.1 of \cite{PIS} and Proposition 4.8 of \cite{Majidesfera}, the diagonal map $\delta_{U(M)}:U(M)\to U(M)\ot U(M)$ defined by $\delta_{U(M)} (x)=1\ot x +x\ot 1$  for all $x\in M,$ and the map $\varepsilon_{U(M)}:U(M)\to K$ defined by   $\varepsilon_{U(M)}(x)=0$   for all $x\in M$, both  extended to $U(M)$ as morphisms of unital magmas; together with the map  $\lambda_{U(M)}:U(M)\to U(M)$, defined by $\lambda_{U(M)}(x)=-x$ for all $x\in M$ and extended to $U(M)$ as an antimultiplicative  morphism, provide a cocommutative Hopf quasigroup structure on  $U(M)$.
}
\end{ejs}

\begin{defin}
\label{H-comodule-magma}
{\rm
Let $H$ be a Hopf quasigroup and let $A$ be a unital  magma (monoid) with a
right coaction $\rho_{A}:A\rightarrow A\ot H$. We will say that
${\Bbb A}=(A,\rho_{A})$ is a right $H$-comodule magma (monoid) if $(A,\rho_{A})$ is
a right $H$-comodule (i.e. $(\rho_{A}\otimes H)\circ
\rho_{A}=(A\otimes \delta_{H})\circ \rho_{A}$, $(A\otimes
\varepsilon_{H})\circ \rho_{A}=id_{A}$), and the following
identities
\begin{itemize}
\item[(b1)] $\rho_{A}\co \eta_{A}=\eta_{A}\ot \eta_{H},$
\item[(b2)] $\rho_{A}\co \mu_{A}=\mu_{A\ot H}\co  (\rho_{A}\ot \rho_{A}),$
\end{itemize}
hold.

Obviously, if $H$ is a Hopf quasigroup, the pair  ${\Bbb H}=(H,\delta_{H})$ is an example of right $H$-comodule magma.

Let ${\Bbb A}$, ${\Bbb B}$ be right $H$-comodule magmas (monoids). A morphism of right $H$-comodule magmas (monoids) $f:{\Bbb A}\rightarrow {\Bbb B}$ is a morphism $f:A\rightarrow B$ in ${\mathcal C}$  of unital magmas (monoids) and  right $H$-comodules, that is $(f\ot H)\co\rho_{A}=\rho_{B}\co f$.
}
\end{defin}

\begin{rem}
\label{isomorphism}
{\rm 
  Note that, if $H$ is cocommutative, every endomorphism $\alpha:{\Bbb H}\rightarrow {\Bbb H}$ of right $H$-comodule magmas is an isomorphism. Indeed: First note that by the comodule condition an the cocommutativity of $H$ we have $\alpha=((\varepsilon_{H}\co \alpha)\ot H)\co \delta_{H}=  (H\ot (\varepsilon_{H}\co \alpha))\co \delta_{H}$ and then $\alpha^{\prime}=(H\ot (\varepsilon_{H}\co \alpha\co \lambda_{H}))\co \delta_{H}$ is the inverse of $\alpha$ because by the properties of $H$:
$$\alpha^{\prime}\co\alpha=\alpha\co\alpha^{\prime}=(H\ot (((\varepsilon_{H}\co \alpha)\ot (\varepsilon_{H}\co \alpha\co \lambda_{H}))\co \delta_{H}))\co \delta_{H}$$
$$=(H\ot (\varepsilon_{H}\co \alpha\co \mu_{H}\co (H\ot \lambda_{H})\co \delta_{H}))\co \delta_{H}=id_{H}.$$

}
\end{rem}

\begin{prop}
\label{tensor-product}  Let $H$ be a Hopf quasigroup and  ${\Bbb A}$, ${\Bbb B}$
right $H$-comodule magmas. The pairs ${\Bbb A}\ot_{1} {\Bbb B}=(A\ot B, \rho_{A\otimes B}^1=(A\ot c_{H,B})\co (\rho_{A}\ot B))$, ${\Bbb A}\ot_{2} {\Bbb B}=(A\ot B, \rho_{A\otimes B}^2=A\ot \rho_{B})$ are right $H$-comodule magmas. Moreover ${\Bbb A}\ot_{1} {\Bbb B}$ and ${\Bbb B}\ot_{2} {\Bbb A}$ are isomorphic right $H$-comodule magmas.
\end{prop}

\begin{proof} We give the proof only for ${\Bbb A}\ot_{1} {\Bbb B}$. The calculus for ${\Bbb A}\ot_{2} {\Bbb B}$ are analogous and we left to  the reader. First note that the object $A\ot B$ is a unital  magma in ${\mathcal C}$. On the other hand, the pair $(A\ot B, \rho_{A\ot B}^1)$ is a right $H$-comodule because trivially $(A\ot B\ot \varepsilon_{H})\co\rho_{A\ot B}^1=id_{A\ot B}$ and using the naturality of $c$ we obtain that $(\rho_{A\ot B}^1\otimes H)\circ \rho_{A\ot B}^1=(A\otimes \delta_{H})\circ \rho_{A\ot B}^1$. Moreover, $\rho_{A\ot B}^1\co \eta_{A\ot B}=\eta_{A\ot B}\ot \eta_{H}$ and also by the naturality of $c$ we have $\rho_{A\ot B}^1\co \mu_{A\ot B}=(\mu_{A\ot B}\ot \mu_{H})\co 
(A\ot B\ot c_{H,A\ot B}\ot H)\co (\rho_{A\ot B}^1\ot \rho_{A\ot B}^1).$
Finally, $c_{A,B}$ is an isomorphism of right $H$-comodule magmas between ${\Bbb A}\ot_{1} {\Bbb B}$ and ${\Bbb B}\ot_{2} {\Bbb A}$ because by the naturally of $c$ we obtain that $c_{A,B}\co \eta_{A\ot B}=\eta_{B\ot A}$, $\mu_{B\ot A}\co ( c_{A,B}\ot c_{A,B})=c_{A,B}\co \mu_{A\ot B}$ and  $\rho_{B\ot A}^2\co c_{A,B}=(c_{A,B}\ot H)\co \rho_{A\ot B}^1$.

\end{proof}

\begin{prop}
\label{op-magma} Let $H$ be a cocommutative Hopf quasigroup and
${\Bbb A}$ a right $H$-comodule magma. Then
$\overline{{\Bbb A}}=(\overline{A}, \rho_{\overline{A}}=(A\ot
\lambda_{H})\co \rho_{A})$ is a right $H$-comodule magma.
\end{prop}

\begin{proof} Trivially $(A\otimes \varepsilon_{H})\co
\rho_{\overline{A}}=id_{A}$. Using that $H$ is cocommutative and
(\ref{anti}) we obtain $(\rho_{\overline{A}}\ot H)\co
\rho_{\overline{A}}=(A\ot \delta_{H})\co \rho_{\overline{A}}$.
Moreover by (b1) of Definition \ref{H-comodule-magma} and
(\ref{fixed}), the identity $\rho_{\overline{A}}\co
\eta_{A}=\eta_{A}\ot \eta_{H}$ holds. Finally, by the naturality of
$c$, (b2) of Definition \ref{H-comodule-magma} and (\ref{anti}) the
equality $\rho_{\overline{A}}\co
\mu_{\overline{A}}=\mu_{\overline{A}\ot H}\co
(\rho_{\overline{A}}\ot \rho_{\overline{A}})$ follows easily.

\end{proof}

\begin{prop}
\label{product} Let $H$ be a Hopf quasigroup and  ${\Bbb A}$, ${\Bbb B}$
right $H$-comodule magmas. The object $A\bullet B$ defined by the equalizer diagram
$$
\setlength{\unitlength}{3mm}
\begin{picture}(22,4)
\put(2,2){\vector(1,0){4}}
\put(12,2.5){\vector(1,0){6}}
\put(12,1.5){\vector(1,0){6}}
\put(0,2){\makebox(0,0){$A\bullet B$}}
\put(9,2){\makebox(0,0){$A\otimes B$}}
\put(22,2){\makebox(0,0){$A\otimes B\otimes H,$}}
\put(4,3){\makebox(0,0){$i_{A\bullet B}$}}
\put(15,4){\makebox(0,0){$\rho_{A\otimes B}^1$}}
\put(15,0.5){\makebox(0,0){$\rho_{A\otimes B}^2$}}
\end{picture}
$$
where $\rho_{A\otimes B}^1$ and $\rho_{A\otimes B}^2$ are the
morphisms defined in Proposition \ref{tensor-product}, is a unital
magma where $\eta_{A\bullet B}$ and $\mu_{A\bullet B}$ are the
factorizations through $i_{A\bullet B}$ of the morphisms $\eta_{A\ot
B}$ and $\mu_{A\ot B}\co (i_{A\bullet B}\ot i_{A\bullet B})$
respectively. Moreover, if $H$ is flat and the coaction
$\rho_{A\bullet B}:A\bullet B\rightarrow A\bullet B\ot H$ is the
factorization of $\rho_{A\otimes B}^2\co i_{A\bullet B}$ through
$i_{A\bullet B}\ot H$, the pair ${\Bbb A}\bullet {\Bbb B}=(A\bullet
B, \rho_{A\bullet B})$ is a right $H$-comodule magma.
\end{prop}

\begin{proof} Trivially $\rho_{A\otimes B}^1\co \eta_{A\ot B}=\eta_{A}\ot \eta_{B}\ot \eta_{H}=\rho_{A\otimes B}^2\co \eta_{A\ot B}$. Therefore, there exists a unique  morphism $\eta_{A\bullet B}:K\rightarrow A\bullet B$ such that $i_{A\bullet B}\co \eta_{A\bullet B}=\eta_{A\ot B}$. On the  other hand, using the properties of $\rho_{A}$ and $\rho_{B}$ and the naturality of $c$ we have
\begin{itemize}
\item[ ]$\hspace{0.38cm}\rho_{A\otimes B}^1\co  \mu_{A\ot B}\co (i_{A\bullet B}\ot i_{A\bullet B})$
\item[ ]$= (\mu_{A\ot B}\ot \mu_{H})\co (A\ot B\ot c_{H,A\ot B}\ot H)\co ((\rho_{A\otimes B}^1\co i_{A\bullet B})\ot  (\rho_{A\otimes B}^1\co i_{A\bullet B}))$
\item[ ]$=(\mu_{A\ot B}\ot \mu_{H})\co (A\ot B\ot c_{H,A\ot B}\ot H)\co ((\rho_{A\otimes B}^2\co i_{A\bullet B})\ot  (\rho_{A\otimes B}^2\co i_{A\bullet B})) $
\item[ ]$=\rho_{A\otimes B}^2\co \mu_{A\ot B}\co (i_{A\bullet B}\ot i_{A\bullet B}).$
\end{itemize}
Then, there exists a unique morphism $\mu_{A\bullet B}:A\bullet B\ot A\bullet B\rightarrow A\bullet B$ such that $i_{A\bullet B}\co \mu_{A\bullet B}=\mu_{A\ot B}\co (i_{A\bullet B}\ot i_{A\bullet B}).$ Moreover,
$\mu_{A\bullet B}\co (\eta_{A\bullet B}\ot A\bullet B)=id_{A\bullet B}=\mu_{A\bullet B}\co (A\bullet B\ot \eta_{A\bullet B})$ because
$i_{A\bullet B}\co \mu_{A\bullet B}\co (\eta_{A\bullet B}\ot A\bullet B)=i_{A\bullet B}= i_{A\bullet B}\co \mu_{A\bullet B}\co (A\bullet B\ot \eta_{A\bullet B}).$ Therefore, $A\bullet B$ is a unital  magma.

Moreover,
$$
\setlength{\unitlength}{3mm}
\begin{picture}(30,4)
\put(4,2){\vector(1,0){5}}
\put(17,2.5){\vector(1,0){6}}
\put(17,1.5){\vector(1,0){6}}
\put(0,2){\makebox(0,0){$A\bullet  B\otimes H$}}
\put(13,2){\makebox(0,0){$A\otimes B\otimes H$}}
\put(28,2){\makebox(0,0){$A\otimes B \otimes H\otimes H$}}
\put(6,3){\makebox(0,0){$\scriptstyle i_{A\bullet  B}\otimes H$}}
\put(20,3.5){\makebox(0,0){$\scriptstyle\rho_{A\otimes B}^1\otimes H$}}
\put(20,0.5){\makebox(0,0){$\scriptstyle\rho_{A\otimes B}^2\otimes H$}}
\end{picture}
$$
is an equalizer diagram, because $-\ot H$ preserves equalizers, and by the the properties of $\rho_{A}$ and $\rho_{B}$ and the naturally of $c$ we obtain $(\rho_{A\otimes B}^1\ot H)\co \rho_{A\otimes B}^2\co i_{A\bullet B}=(\rho_{A\otimes B}^2\ot H)\co \rho_{A\otimes B}^2\co i_{A\bullet B}.$ As a consequence, there exists a unique morphism $\rho_{A\bullet B}:A\bullet B\rightarrow A\bullet B\ot H$ such that $(i_{A\bullet  B}\otimes H)\co \rho_{A\bullet B}=\rho_{A\otimes B}^2\co i_{A\bullet B}$. Then, the pair $(A\bullet B, \rho_{A\bullet B})$ is a right $H$-comodule because $(i_{A\bullet  B}\ot \varepsilon_{H})\co \rho_{A\bullet B}=i_{A\bullet  B}$ and also $(((i_{A\bullet  B}\ot H)\co \rho_{A\bullet B})\ot H)\co \rho_{A\bullet B}=(i_{A\bullet  B}\ot \delta_{H})\co \rho_{A\bullet B}.$ Finally,  (b1) and (b2) of Definition \ref{H-comodule-magma} follow, by a similar reasoning, from $(i_{A\bullet  B}\ot H)\co \rho_{A\bullet B}\co \eta_{A\bullet B}=
(i_{A\bullet  B}\ot H)\co (\eta_{A\bullet B}\ot \eta_{H})$ and  $(i_{A\bullet  B}\ot H)\co \rho_{A\bullet B}\co \mu_{A\bullet B}= (i_{A\bullet  B}\ot H)\co \mu_{A\bullet  B\ot H}\co (\rho_{A\bullet B}\ot \rho_{A\bullet B})$.
\end{proof}

\begin{prop}
\label{morphisms} Let $H$ be a flat Hopf quasigroup and $f:{\Bbb
A}\rightarrow {\Bbb B}$, $g:{\Bbb T}\rightarrow {\Bbb D}$ morphisms
of right $H$-comodule magmas. Then the morphism $f\bullet g:A\bullet
T\rightarrow B\bullet D$, obtained as the factorization of $(f\ot
g)\co i_{A\bullet T}:A\bullet T\rightarrow B\otimes D$ through the
equalizer $i_{B\bullet D}$, is a morphism of right $H$-comodule
magmas between ${\Bbb A}\bullet {\Bbb T}$ and ${\Bbb B}\bullet {\Bbb
D}$. Moreover, if $f$ and $g$ are isomorphisms, so is $f\bullet g$.
\end{prop}

\begin{proof} Using that  $f$ and $g$ are comodule morphisms we
obtain $\rho_{B\ot D}^{1}\co (f\ot g)\co i_{A\bullet T}=\rho_{B\ot D}^{2}\co
(f\ot g)\co i_{A\bullet T}$ and as a consequence there exist a
unique morphism $(f\bullet g): A\bullet T\rightarrow B\bullet D$
such that $i_{B\bullet D}\co (f\bullet g)=(f\ot g)\co i_{A\bullet
T}$. The morphism $f\bullet g$ is a morphism of unital magmas
because $i_{B\bullet D}\co \eta_{B\bullet D}= i_{B\bullet D}\co
(f\bullet g)\co \eta_{A\bullet T}$ and for the product the equality
$i_{B\bullet D}\co \mu_{B\bullet D}\co ((f\bullet g)\ot (f\bullet
g))=i_{B\bullet D}\co (f\bullet g)\co \mu_{A\bullet T}$ holds. Also,
it is a comodule morphism because $(i_{B\bullet D}\ot H)\co
\rho_{B\bullet D}\co (f\bullet g)=(i_{B\bullet D}\ot H)\co
((f\bullet g)\ot H)\co \rho_{A\bullet T}.$

Finally, it is easy to show that, if $f$ and $g$ are isomorphisms,
$f\bullet g$ is an isomorphism with inverse $f^{-1}\bullet g^{-1}$.
\end{proof}

\begin{prop}
\label{product-commutative} Let $H$ be a flat Hopf quasigroup and  ${\Bbb A}$, ${\Bbb B}$
right $H$-comodule magmas. Then  ${\Bbb A}\bullet {\Bbb B}$ and ${\Bbb B}\bullet {\Bbb A}$ are isomorphic as right $H$-comodule magmas.
\end{prop}

\begin{proof} First note that by the naturally of $c$ and the properties of the equaliser morphism $i_{A\bullet B}$ we have that $\rho_{B\otimes A}^1\co c_{A,B}\co i_{A\bullet B}=\rho_{B\otimes A}^2\co c_{A,B}\co i_{A\bullet B}$ and then there exist    a morphism $\tau_{A,B}:A\bullet B\rightarrow B\bullet A$ such that $i_{B\bullet A}\co \tau_{A,B}=c_{A,B}\co i_{A\bullet B}$. Also there exists an unique morphism $\tau_{B,A}:B\bullet A\rightarrow A\bullet B$ such that $i_{A\bullet B}\co \tau_{B,A}=c_{B,A}\co i_{B\bullet A}$. Then $i_{A\bullet B}\co \tau_{B,A}\co \tau_{A,B}=c_{B,A}\co c_{A,B}\co i_{A\bullet B}=i_{A\bullet B}$ and  similarly $i_{B\bullet A}\co \tau_{A,B}\co \tau_{B,A}=i_{B\bullet A}$. Thus  $\tau_{A,B}$ is an isomorphism with inverse $\tau_{B,A}$. Moreover, $i_{B\bullet A}\co \tau_{A,B}\co \eta_{A\bullet B}=c_{A,B}\co i_{A\bullet B}\co \eta_{A\bullet B}=\eta_{B\ot A}=i_{B\bullet A}\co \eta_{B\bullet A}$ and
$$i_{B\bullet A}\co \tau_{A,B}\co \mu_{A\bullet B}=\mu_{B\ot A}\co ((c_{A,B}\co i_{A\bullet B})\ot (c_{A,B}\co i_{A\bullet B}))=\mu_{B\ot A}\co ((i_{B\bullet A}\co \tau_{A,B})\ot (i_{B\bullet A}\co \tau_{A,B}))$$
$$=i_{B\bullet A}\co \mu_{B\bullet A}\co (\tau_{A,B}\ot \tau_{A,B}).$$
Therefore, $\tau_{A,B}$ is a morphism of unital magmas and finally it is a morphism of right $H$-comodules because $((i_{B\bullet A}\co \tau_{A,B})\ot H)\co \rho_{A\bullet B}=(i_{B\bullet A}\ot H)\co \rho_{B\bullet A}\co \tau_{A,B}$.

\end{proof}

\begin{prop}
\label{product-associative} Let $H$ be a flat  Hopf quasigroup and
${\Bbb A}$, ${\Bbb B}$, ${\Bbb
D}$ right $H$-comodule magmas such that $A$ and $D$ are
flat. Then  ${\Bbb A}\bullet ({\Bbb B}\bullet {\Bbb D})$ and $({\Bbb
A}\bullet {\Bbb B})\bullet {\Bbb D}$ are isomorphic as right
$H$-comodule magmas.
\end{prop}

\begin{proof}

First, note that
$$
 \setlength{\unitlength}{3mm}
\begin{picture}(45,4)
\put(11,2){\vector(1,0){6}} \put(26,2.5){\vector(1,0){7}}
\put(26,1.5){\vector(1,0){7}} \put(7,2){\makebox(0,0){$A\ot B\bullet
D$}}

\put(22,2){\makebox(0,0){$A\ot B\ot D$}}
\put(38,2){\makebox(0,0){$A\ot B\ot D\otimes H$}}
\put(13,3){\makebox(0,0){$\scriptstyle A\ot i_{B \bullet D}$}}
\put(30,3.5){\makebox(0,0){$\scriptstyle A\ot \rho_{B\ot D}^1$}}
\put(30,0.5){\makebox(0,0){$\scriptstyle A\ot \rho_{B\ot D}^2$}}
\end{picture}
$$
and
$$
 \setlength{\unitlength}{3mm}
\begin{picture}(45,4)
\put(10,2){\vector(1,0){6}} \put(24,2.5){\vector(1,0){11}}
\put(24,1.5){\vector(1,0){11}} \put(7,2){\makebox(0,0){$A\bullet
B\ot D$}}

\put(20,2){\makebox(0,0){$A\ot B\ot D$}}
\put(40,2){\makebox(0,0){$A\ot B\ot D\ot H$}}
\put(13,3){\makebox(0,0){$\scriptstyle  i_{A \bullet B}\ot D$}}
\put(29,3.5){\makebox(0,0){$\scriptstyle (A\ot B\ot c_{H,D} )\co
(\rho_{A\ot B}^1\ot D)$}} \put(29,0.5){\makebox(0,0){$\scriptstyle
(A\ot B\ot c_{H,D} )\co (\rho_{A\ot B}^2\ot D)$}}
\end{picture}
$$
are equalizer diagrams because $A$ and $D$ are flat and $A\ot B\ot
c_{H,D}$ an isomorphism. On the other hand, it is  easy to show that
$$(A\ot i_{B\bullet D}\ot H)\co \rho_{A\ot B\bullet D}^1=
(A\ot B\ot c_{H,D} )\co (\rho_{A\ot B}^1\ot D)\co (A\ot i_{B\bullet
D})$$ and
$$(A\ot i_{B\bullet D}\ot H)\co \rho_{A\ot B\bullet D}^2=
(A\ot B\ot c_{H,D} )\co (\rho_{A\ot B}^2\ot D)\co (A\ot i_{B\bullet
D}).$$ Therefore
$$(A\ot B\ot c_{H,D} )\co (\rho_{A\ot B}^1\ot D)\co (A\ot i_{B\bullet
D})\co  i_{A\bullet (B\bullet D)}=(A\ot B\ot c_{H,D} )\co
(\rho_{A\ot B}^2\ot D)\co (A\ot i_{B\bullet D})\co  i_{A\bullet
(B\bullet D)}$$ and as a consequence there exist a unique morphism
$h:A\bullet (B\bullet D)\rightarrow (A\bullet B)\ot D$ such that
\begin{equation}
\label{h} (i_{A\bullet B}\ot D)\co h=(A\ot i_{B\bullet D})\co
i_{A\bullet (B\bullet D)}.
\end{equation}

The diagram
$$
 \setlength{\unitlength}{3mm}
\begin{picture}(45,4)
\put(11,2){\vector(1,0){6}} \put(26,2.5){\vector(1,0){7}}
\put(26,1.5){\vector(1,0){7}}

\put(7,2){\makebox(0,0){$A\bullet (B\bullet D)$}}

\put(22,2){\makebox(0,0){$(A\bullet  B)\otimes D$}}
\put(38,2){\makebox(0,0){$A\bullet B\ot D\otimes H$}}
\put(14,3){\makebox(0,0){$ h$}}
\put(30,3.5){\makebox(0,0){$\scriptstyle  \rho_{A\bullet B\ot
D}^1$}} \put(30,0.5){\makebox(0,0){$\scriptstyle \rho_{A\bullet B\ot
D}^2$}}
\end{picture}
$$
is an equalizer diagram. Indeed,  it is easy to see that $\rho_{A\bullet B\ot
D}^1\co h=\rho_{A\bullet B\ot
D}^1\co h$ and, if $f:C\rightarrow A\bullet B\ot D$
is a morphism such that $\rho_{(A\bullet B)\ot D}^{1}\co
f=\rho_{(A\bullet B)\ot D}^{2}\co f$, we have that $$(A\ot
\rho_{B\ot D}^1)\co (i_{A \bullet B}\ot D)\co f=(A\ot \rho_{B\ot
D}^2)\co (i_{A \bullet B}\ot D)\co f$$ because
$$(A\ot \rho_{B\ot D}^1)\co (i_{A \bullet
B}\ot D)=(i_{A \bullet B}\ot D\ot H)\co \rho_{A\bullet B\ot D}^1$$
and
$$(A\ot \rho_{B\ot D}^2)\co (i_{A \bullet
B}\ot D)=(i_{A \bullet B}\ot D\ot H)\co \rho_{A\bullet B\ot D}^2.$$
Then, there exists a unique morphism $t:C\rightarrow A\ot B\bullet D$
such that  $(A\ot i_{B\bullet D})\co t=(i_{A \bullet B}\ot D)\co f$.
The morphism $t$ factorizes through the equalizer $i_{A\bullet (B\bullet D)}$ because 
$$(A\ot i_{B \bullet D}\ot H)\co \rho_{A\ot B\bullet D}^{1}\co t=
(A\ot i_{B \bullet D}\ot H)\co \rho_{A\ot B\bullet D}^{2}\co t$$ and then 
$$\rho_{A\ot B\bullet D}^{1}\co t=
 \rho_{A\ot B\bullet D}^{2}\co t$$
holds. Thus, there exists a unique morphism $g:C\rightarrow A\bullet
(B\bullet D)$ satisfying the equality $i_{A\bullet (B\bullet D)}\co
g=t$. As a consequence
$$(i_{A\bullet B}\ot D)\co h\co g=(A\ot i_{B\bullet D})\co  i_{A\bullet
(B\bullet D)}\co g=(A\ot i_{B\bullet D})\co t=(i_{A \bullet B}\ot
D)\co f$$ and then $h\co g=f$. Moreover, $g$ is the unique morphism
such that $h\co g=t$, because if $d:C\rightarrow A\bullet (B\bullet
D)$ satisfies $h\co d=f$, we obtain that $i_{A\bullet (B\bullet
D)}\co d=t$ and therefore $d=g$.

As a cosequence, there exists an isomorphism $n_{A,B,C}:A\bullet (B\bullet
D)\rightarrow (A\bullet B)\bullet D$ such that
\begin{equation}
\label{inv-ass-const} i_{(A\bullet B)\bullet D}\co n_{A,B,D}=h
\end{equation}

The isomorphism $n_{A,B,C}$ is a morphism of unital magmas because
by (\ref{h}), (\ref{inv-ass-const}) and the naturality of $c$ we
have

\begin{itemize}
\item[ ]$\hspace{0.38cm}(i_{A\bullet B}\ot D)\co i_{(A\bullet B)\bullet D}\co n_{A,B,D}\co
\eta_{A\bullet (B\bullet D)} $
\item[ ]$=(i_{A\bullet B}\ot D)\co h\co \eta_{A\bullet (B\bullet D)} $
\item[ ]$=(A\ot i_{B\bullet D})\co
i_{A\bullet (B\bullet D)}\co \eta_{A\bullet (B\bullet D)} $
\item[ ]$=\eta_{A}\ot \eta_{B}\ot \eta_{D} $
\item[ ]$=(i_{A\bullet B}\ot D)\co i_{(A\bullet B)\bullet D}\co \eta_{(A\bullet B)\bullet D} $
\end{itemize}
and
\begin{itemize}
\item[ ]$\hspace{0.38cm}(i_{A\bullet B}\ot D)\co i_{(A\bullet B)\bullet D}\co n_{A,B,D}\co \mu_{A\bullet (B\bullet D)}  $
\item[ ]$=(i_{A\bullet B}\ot D)\co h \co \mu_{A\bullet (B\bullet D)}  $
\item[ ]$=(A\ot i_{B\bullet D})\co
i_{A\bullet (B\bullet D)}\co \mu_{A\bullet (B\bullet D)} $
\item[ ]$=(A\ot i_{B\bullet D})\co \mu_{A\ot (B\bullet D)} \co
(i_{A\bullet (B\bullet D)}\ot i_{A\bullet (B\bullet D)})$
\item[ ]$= \mu_{A\ot B\ot D}\co (((A\ot i_{B\bullet D})\co
i_{A\bullet (B\bullet D)})\ot ((A\ot i_{B\bullet D})\co i_{A\bullet
(B\bullet D)})) $
\item[ ]$=\mu_{A\ot B\ot D}\co (((i_{A\bullet B}\ot D)\co h)\ot
((i_{A\bullet B}\ot D)\co h)) $
\item[ ]$=(i_{A\bullet B}\ot D)\co \mu_{A\bullet B\ot D}\co (h\ot h) $
\item[ ]$= (i_{A\bullet B}\ot D)\co \mu_{A\bullet B\ot D}\co
((i_{(A\bullet B)\bullet D}\co n_{A,B,C})\ot (i_{(A\bullet B)\bullet D}\co n_{A,B,C})) $
\item[ ]$=(i_{A\bullet B}\ot D)\co i_{(A\bullet B)\bullet D}\co \mu_{(A\bullet B)\bullet D}\co
(n_{A,B,D}\ot n_{A,B,D}) $
\end{itemize}
Finally, using a similar reasoning, we obtain that $n_{A,B,C}$ is a
morphism of right $H$-comodules because
\begin{itemize}
\item[ ]$\hspace{0.38cm}(i_{A\bullet B}\ot D\ot H)\co (i_{(A\bullet B)\bullet D}\ot H)\co
\rho_{(A\bullet B)\bullet D}\co n_{A,B,D} $
\item[ ]$=(A\ot B\ot \rho_{D})\co (i_{A\bullet B}\ot D)\co i_{(A\bullet B)\bullet D}\co n_{A,B,C} $
\item[ ]$= (A\ot B\ot \rho_{D})\co (i_{A\bullet B}\ot D)\co h$
\item[ ]$=(A\ot B\ot \rho_{D})\co (A\ot i_{B\bullet D})\co
i_{A\bullet (B\bullet D)}$
\item[ ]$=(A\ot i_{B\bullet D}\ot H)\co (A\ot \rho_{B\bullet D})\co i_{A\bullet (B\bullet D)}  $
\item[ ]$=(A\ot i_{B\bullet D}\ot H)\co (i_{A\bullet (B\bullet D)}\ot H)\co \rho_{A\bullet (B\bullet D)} $
\item[ ]$=(i_{A\bullet B}\ot D\ot H)\co (h \ot H)\co \rho_{A\bullet (B\bullet D)}$
\item[ ]$= (i_{A\bullet B}\ot D\ot H)\co (i_{(A\bullet B)\bullet D}\ot H)\co
(n_{A,B,D}\ot H)\co \rho_{A\bullet (B\bullet D)}.$
\end{itemize}

\end{proof}

\begin{prop}
\label{product-unit}  Let $H$ be a cocommutative  Hopf quasigroup
and  ${\Bbb A}$ a right $H$-comodule magma. Then
\begin{equation}
\label{equalizer} \setlength{\unitlength}{3mm}
\begin{picture}(22,5)
\put(2,2){\vector(1,0){4}} \put(12,2.5){\vector(1,0){6}}
\put(12,1.5){\vector(1,0){6}} \put(0,2){\makebox(0,0){$A$}}
\put(9,2){\makebox(0,0){$A\otimes H$}}
\put(22,2){\makebox(0,0){$A\otimes H\otimes H,$}}
\put(4,3){\makebox(0,0){$\rho_{A}$}}
\put(15,4){\makebox(0,0){$\rho_{A\otimes H}^1$}}
\put(15,0.5){\makebox(0,0){$\rho_{A\otimes H}^2$}}
\end{picture}
\end{equation}
is an equalizer diagram. If $H$ is flat ${\Bbb A}\bullet {\Bbb H}$
and ${\Bbb A}$ are isomorphic as right $H$-comodule magmas.
\end{prop}

\begin{proof} We will begin by showing that 
 (\ref{equalizer}) is an equalizer diagram. Indeed, if $H$ is
cocommutative we have that $\rho_{A\otimes H}^1\co \rho_{A}=(A\ot
(c_{H,H}\co \delta_{H}))\co \rho_{A}=\rho_{A\otimes H}^2\co
\rho_{A}.$ Moreover, if there exists a morphism $f:D\rightarrow A\ot
H$ such that $\rho_{A\otimes H}^1\co f=\rho_{A\otimes H}^2\co f$, we
have that $\rho_{A}\co (A\ot \varepsilon_{H})\co f=f$ and if
$g:D\rightarrow A$ is a morphism such that $\rho_{A}\co g=f$ this
gives $g=(A\ot \varepsilon_{H})\co f$.  Therefore there is a 
unique isomorphism $r_{A}:A\bullet H\rightarrow A$ satisfying
$\rho_{A}\co r_{A}=i_{A\bullet H}$.

In the second step we show that $r_{A}$ is a morphism of right $H$-comodule magmas. Trivially, $r_{A}\co \eta_{A\bullet H}=\eta_{A}$ because  $\rho_{A}\co r_{A}\co \eta_{A\bullet H}=i_{A\bullet H} \co \eta_{A\bullet H}=\eta_{A}\ot \eta_{H}=\rho_{A}\co \eta_{A}$. Also $\mu_{A}\co (r_{A}\ot r_{A})=r_{A}\co \mu_{A\bullet H}$ because
$\rho_{A}\co \mu_{A}\co (r_{A}\ot r_{A})=\rho_{A}\co r_{A}\co \mu_{A\bullet H}$ and as a consequence $r_{A}$ is a morphism of unital magmas. Finally, the $H$-comodule condition follows from $((\rho_{A}\co r_{A})\ot H)\co \rho_{A\bullet H}=(\rho_{A}\ot H)\co \rho_{A}\co r_{A}$.
\end{proof}

\begin{rem}
\label{rem-product-unit}
{\rm Note that, under the conditions of the previous proposition, the coaction for ${\Bbb A}\bullet {\Bbb H}$ is $i_{A\bullet B}$.  On the other hand, Proposition \ref{product-unit} gives that
$$
\setlength{\unitlength}{3mm}
\begin{picture}(22,4)
\put(2,2){\vector(1,0){4}}
\put(12,2.5){\vector(1,0){6}}
\put(12,1.5){\vector(1,0){6}}
\put(0,2){\makebox(0,0){$H$}}
\put(9,2){\makebox(0,0){$H\otimes H$}}
\put(22,2){\makebox(0,0){$A\otimes H\otimes H,$}}
\put(4,3){\makebox(0,0){$\delta_{H}$}}
\put(15,4){\makebox(0,0){$\rho_{H\otimes H}^1$}}
\put(15,0.5){\makebox(0,0){$\rho_{H\otimes H}^2$}}
\end{picture}
$$ is an equalizer diagram.

}
\end{rem}

\begin{prop}
\label{monoidal-magmas}
 Let $H$ be a cocommutative  Hopf quasigroup and $Mag_{f}({\mathcal C}, H)$ be the category whose objects are flat $H$-comodule
magmas and whose arrows are the morphism of $H$-comodule magmas. Then  $Mag_{f}({\mathcal C}, H)$
is a symmetric monoidal category.
\end{prop}

\begin{proof} The category $Mag_{f}({\mathcal C}, H)$ is a monoidal
category with the tensor product defined by the product $"\bullet "
$ introduced in Proposition \ref{product}, with unit ${\Bbb H}$,
with associative constraints $\mathfrak{a}_{{\Bbb A},{\Bbb B},{\Bbb
D}}=n_{A,B,D}^{-1}$, where $n_{A,B,D}$ is the isomorphism defined in
Proposition \ref{product-associative}, and right unit constraints
and left unit constraints $\mathfrak{r}_{A}=r_{A}$,
$\mathfrak{l}_{A}=r_{A}\co \tau_{H,A}$ respectively, where $r_{A}$
is the isomorphism defined in Proposition \ref{product-unit} and
$\tau_{H,A}$ the one defined in Proposition
\ref{product-commutative}. It is easy but tedious, and we leave the
details to the reader, to show that  associative constraints and
 right and left unit constraints are natural and satisfy the
Pentagon Axiom and the Triangle Axiom. Finally the tensor product of
two morphisms is defined by Proposition \ref{morphisms} and, of
course, the symmetry isomorphism is the transformation $\tau$
defined in Proposition \ref{product-commutative}.
\end{proof}

\section{The group of strong Galois objects}

The aim of this section is to introduce the notion of  strong
Galois $H$-object for a cocommutative Hopf quasigroup $H$. We will
prove that the set of isomorphism classes of strong $H$-Galois
objects is a group that becomes  the classical Galois group when $H$
is a cocommutative Hopf algebra.

\begin{defin}
\label{Galois} {\rm Let $H$ be a Hopf quasigroup and ${\Bbb
A}$ a right $H$-comodule magma. We will say that ${\Bbb
A}$ is a  Galois $H$-object if
\begin{itemize}
\item[(c1)] $A$ is faithfully flat.
\item[(c2)] The canonical morphism $\gamma_{A}=(\mu_{A}\ot H)\co
(A\ot \rho_{A}):A\ot A\rightarrow A\ot H$ is an isomorphism.
\end{itemize}
If moreover, $f_{A}=\gamma_{A}^{-1}\co (\eta_{A}\ot H):H\rightarrow
A^e$ is a morphism of unital magmas, we will say that ${\Bbb A}$ is
a strong Galois $H$-object.

A morphism between to (strong) Galois $H$-objects is  a morphism of
right $H$-comodule magmas.

 Note that if ${\Bbb A}$ is a strong Galois $H$-object and ${\Bbb B}$ is a Galois $H$-object isomorphic to ${\Bbb A}$ as Galois $H$-objects then ${\Bbb B}$ is also a strong Galois $H$-object because, if $g:A\rightarrow B$ is the isomorphism, we have $\gamma_{B}\co (g\ot g)=(g\ot H)\co \gamma_{A}$ and it follows that $f_{B}=
(g\ot g)\co f_{A}$. Then, $f_{B}$ is a morphism of unital magmas and ${\Bbb B}$ is strong.

}
\end{defin}

\begin{ej}
\label{H-strong} {\rm If $H$ is a faithfully flat Hopf quasigroup,
${\Bbb H}$ is a strong Galois $H$-object because
$\gamma_{H}=(\mu_{H}\ot H)\co (H\ot \delta_{H})$ is an isomorphism
with inverse $\gamma_{H}^{-1}=((\mu_{H}\co (H\ot \lambda_{H}))\ot
H)\co (H\ot \delta_{H})$ and  $f_{H}=(\lambda_{H}\ot H)\co
\delta_{H}:H\rightarrow H^e$ is a morphism of unital magmas. }
\end{ej}

\begin{rem}
\label{Hopf-sgal=gal} {\rm If $H$ is a Hopf algebra and ${\Bbb
A}$ is a right $H$-comodule monoid, we say that ${\Bbb
A}$ is a Galois $H$-object when  $A$ is faithfully flat  and  the
canonical morphism $\gamma_{A}$ is an isomorphism. In this setting
every Galois $H$-object is a strong Galois $H$-object because
\begin{itemize}
\item[ ]$\hspace{0.38cm}\gamma_{A}\co \mu_{A^{e}}\co
((\gamma_{A}^{-1}\co (\eta_{A}\ot H))\ot (\gamma_{A}^{-1}\co (\eta_{A}\ot H)) )$
\item[ ]$=(\mu_{A}\ot H)\co (A\ot \mu_{A\ot H})\co (A\ot (\gamma_{A}\co \gamma_{A}^{-1}\co (\eta_{A}\ot H))
\ot \rho_{A})\co (c_{H,A}\ot A)\co (H\ot (\gamma_{A}^{-1}\co
(\eta_{A}\ot H)))$
\item[ ]$=(A\ot \mu_{H})\co (c_{H,A}\ot H)\co (H\ot (\gamma_{A}\co \gamma_{A}^{-1}\co (\eta_{A}\ot H))) $
\item[ ]$=\eta_{A}\ot \mu_{H} $
\item[ ]$= \gamma_{A}\co \gamma_{A}^{-1}\co (\eta_{A}\ot \mu_{H}) $
\end{itemize}
where the  equalities follow by (b2) of Definition
\ref{H-comodule-magma}, the naturality of $c$ and the associativity
of $\mu_{A}$. 
}
\end{rem}

\begin{prop}
\label{trivial-coinvaiants} Let $H$ be a Hopf quasigroup and  ${\Bbb
A}$ a Galois $H$-object. Then
\begin{equation}
\label{coinvariants=K}
 \setlength{\unitlength}{3mm}
\begin{picture}(30,4)
\put(3,2){\vector(1,0){4}} \put(11,2.5){\vector(1,0){10}}
\put(11,1.5){\vector(1,0){10}} \put(1,2){\makebox(0,0){$K$}}
\put(24,2){\makebox(0,0){$A\otimes H$}}
\put(5.5,3){\makebox(0,0){$\eta_{A}$}}
\put(16,3.5){\makebox(0,0){$\rho_{A}$}}
\put(16,0.5){\makebox(0,0){$A\ot \eta_{H}$}}
\put(9,2){\makebox(0,0){$A$}}
\end{picture}
\end{equation}
is an equalizer diagram.
\end{prop}

\begin{proof}
First note that
$$
\setlength{\unitlength}{3mm}
\begin{picture}(30,4)
\put(3,2){\vector(1,0){3.5}} \put(11,2.5){\vector(1,0){10}}
\put(11,1.5){\vector(1,0){10}} \put(1,2){\makebox(0,0){$A$}}
\put(25,2){\makebox(0,0){$A\ot A\otimes A$}}
\put(5,3){\makebox(0,0){$A\ot \eta_{A}$}}
\put(16,3.5){\makebox(0,0){$A\ot A\ot \eta_{A}$}}
\put(16,0.5){\makebox(0,0){$A\ot \eta_{A}\ot A$}}
\put(8.5,2){\makebox(0,0){$A\ot A$}}
\end{picture}
$$
is an equalizer diagram. Then, using that $A$  s faithfully flat, so
is
$$
\setlength{\unitlength}{3mm}
\begin{picture}(30,4)
\put(3,2){\vector(1,0){4}} \put(11,2.5){\vector(1,0){10}}
\put(11,1.5){\vector(1,0){10}} \put(1,2){\makebox(0,0){$K$}}
\put(24,2){\makebox(0,0){$A\otimes A$}}
\put(5.5,3){\makebox(0,0){$\eta_{A}$}}
\put(16,3.5){\makebox(0,0){$A\ot \eta_{A}$}}
\put(16,0.5){\makebox(0,0){$\eta_{A}\ot A$}}
\put(9,2){\makebox(0,0){$A$}}
\end{picture}
$$
 On the other hand, $\gamma_{A}\co (A\ot \eta_{A})=A\ot \eta_{H},\;$ $\gamma_{A}\co
(\eta_{A}\ot A)=\rho_{A}.$ Therefore, if $\gamma_{A}$ is an
isomorphism, (\ref{coinvariants=K}) is an equalizer diagram.

\end{proof}

\begin{lem}
\label{lemma-2}  Let $H$ be a Hopf quasigroup and let ${\Bbb
A}$ a Galois $H$-object. The following equalities hold:
\begin{itemize}
\item[(i)] $\rho_{A\ot A}^2\co \gamma_{A}^{-1}=(\gamma_{A}^{-1}\ot
H)\co (A\ot \delta_{H}).$
\item[(ii)] $\rho_{A\ot A}^1\co
\gamma_{A}^{-1}=(\gamma_{A}^{-1}\ot H)\co (A\ot c_{H,H})\co (A\ot
\mu_{H}\ot H)\co (\rho_{A}\ot ((\lambda_{H}\ot H)\co \delta_{H})).$
\end{itemize}

\end{lem}

\begin{proof} The proof for (i) follows from the identity
$(\gamma_{A}\ot H)\co \rho_{A\ot A}^2=(A\ot \delta_{H})\co \gamma_{A}$.
To obtain (ii), first we prove that
\begin{equation}
\label{tech-ii} (A\ot \mu_{H}\ot H)\co (\rho_{A}\ot ((\lambda_{H}\ot
H)\co \delta_{H}))\co \gamma_{A}=(A\ot c_{H,H})\co (\gamma_{A}\ot
H)\co \rho_{A\ot A}^1.
\end{equation}
Indeed, 
\begin{itemize}
\item[ ]$\hspace{0.38cm}(A\ot \mu_{H}\ot H)\co (\rho_{A}\ot ((\lambda_{H}\ot
H)\co \delta_{H}))\co \gamma_{A}  $
\item[ ]$= (A\ot (\mu_H\circ(\mu_H\ot \lambda_H)\circ (H\ot
\delta_H))\ot H)\co (\mu_{A}\ot H\ot \delta_{H})\co (A\ot c_{H,A}\ot
H)\co (\rho_{A}\ot \rho_{A})$
\item[ ]$= (\mu_{A}\ot H\ot ((\varepsilon_{H}\ot H)\co \delta_{H}))\co (A\ot c_{H,A}\ot
H)\co (\rho_{A}\ot \rho_{A})$
\item[ ]$= (\mu_{A}\ot H \ot H)\co (A\ot c_{H,A}\ot
H)\co (\rho_{A}\ot \rho_{A})  $
\item[ ]$= (\mu_{A}\ot c_{H,H})\co (\rho_{A\ot A}^2\ot H)\co \rho_{A\ot A}^1$
\item[ ]$=(A\ot c_{H,H})\co (\gamma_{A}\ot
H)\co \rho_{A\ot A}^1.$
\end{itemize}
where the first equality follows by the comodule condition for $A$,
the second one by (a2-2) of Definition \ref{Hopf quasigroup}, the
third one by the counit condition, the fourth and the last ones by
the symmetry of $c$ and   the naturality of the braiding.

Then, by (\ref{tech-ii}) we obtain
\begin{itemize}
\item[ ]$\hspace{0.38cm}\rho_{A\ot A}^1\co
\gamma_{A}^{-1}  $
\item[ ]$= (\gamma_{A}^{-1}\ot H)\co (A\ot (c_{H,H}\co
c_{H,H}))\co (\gamma_{A}\ot H)\co \rho_{A\ot A}^1\co \gamma_{A}^{-1}$
\item[ ]$=(\gamma_{A}^{-1}\ot H)\co (A\ot c_{H,H})\co (A\ot
\mu_{H}\ot H)\co (\rho_{A}\ot ((\lambda_{H}\ot H)\co \delta_{H}))$
\end{itemize}
and (ii) holds.
\end{proof}

\begin{prop}
\label{product-Gal} Let $H$ be a cocommutative faithfully flat Hopf quasigroup.
The following assertions hold:
\begin{itemize}
\item[(i)] If ${\Bbb A}$ and ${\Bbb B}$ are Galois $H$-objects so is ${\Bbb
A}\bullet {\Bbb B}$.
\item[(ii)] If ${\Bbb A}$ and ${\Bbb B}$ are strong Galois $H$-objects so is ${\Bbb
A}\bullet {\Bbb B}$.
\end{itemize}
\end{prop}

\begin{proof} First we prove (i). Let ${\Bbb A}$ and ${\Bbb B}$ be
Galois $H$-objects. By Proposition \ref{product} we know that
$A\bullet B$ is a unital magma where $\eta_{A\bullet B}$ and
$\mu_{A\bullet B}$ are the factorizations through $i_{A\bullet B}$
of the morphisms $\eta_{A\ot B}$ and $\mu_{A\ot B}\co (i_{A\bullet
B}\ot i_{A\bullet B})$ respectively. Moreover, using that $H$ is
flat we have that $A\bullet B$ is a right $H$-comodule magma where
the coaction $\rho_{A\bullet B}:A\bullet B\rightarrow A\bullet B\ot
H$ is the factorization of $\rho_{A\otimes B}^2\co i_{A\bullet B}$
(or $\rho_{A\otimes B}^1\co i_{A\bullet B}$) through $i_{A\bullet
B}\ot H$.

The objects $A$ and $B$ are faithfully flat and then so is $A\ot B$. Therefore
$$
 \setlength{\unitlength}{3mm}
\begin{picture}(45,4)
\put(11,2){\vector(1,0){6}} \put(27,2.5){\vector(1,0){7}}
\put(27,1.5){\vector(1,0){7}} \put(6,2){\makebox(0,0){$A\ot B\ot
A\bullet B$}} \put(22,2){\makebox(0,0){$A\ot B\ot A\ot B$}}
\put(40,2){\makebox(0,0){$A\ot B\ot A\ot B \otimes H$}}
\put(14,3){\makebox(0,0){$\scriptstyle A\ot B\ot i_{A \bullet B}$}}
\put(30,3.5){\makebox(0,0){$\scriptstyle A\ot B\ot  \rho_{A\ot
B}^1$}} \put(30,0.5){\makebox(0,0){$\scriptstyle A\ot B\ot
\rho_{A\ot B}^2$}}
\end{picture}
$$
is an equalizer diagram. On the other hand, if $H$ is cocommutative
$$
 \setlength{\unitlength}{3mm}
\begin{picture}(45,4)
\put(11,2){\vector(1,0){6}} \put(27,2.5){\vector(1,0){7}}
\put(27,1.5){\vector(1,0){7}} \put(6,2){\makebox(0,0){$A\ot B\ot
H$}} \put(22,2){\makebox(0,0){$A\ot B\ot H\ot H$}}
\put(40,2){\makebox(0,0){$A\ot B\ot H\ot H \otimes H$}}
\put(14,3){\makebox(0,0){$\scriptstyle A\ot B\ot \delta_H$}}
\put(30,3.5){\makebox(0,0){$\scriptstyle A\ot B\ot  \rho_{H\ot
H}^1$}} \put(30,0.5){\makebox(0,0){$\scriptstyle A\ot B\ot
\rho_{H\ot H}^2$}}
\end{picture}
$$
 is an equalizer diagram (see Remark \ref{rem-product-unit}).

Let $\Gamma_{A\ot B}:A\ot B\ot A\ot B\rightarrow A\ot B\ot H\ot H$
be the morphism defined by
$$\Gamma_{A\ot B}=(A\ot c_{H,B}\ot H)\co (\gamma_{A}\ot
\gamma_{B})\co (A\ot c_{B,A}\ot B).$$ Trivially $\Gamma_{A\ot B}$ is
an isomorphism wit inverse
$$\Gamma_{A\ot B}^{-1}=(A\ot c_{A,B}\ot B)\co (\gamma_{A}^{-1}\ot
\gamma_{B}^{-1})\co (A\ot c_{B,H}\ot H)$$ and satifies
\begin{itemize}
\item[ ]$\hspace{0.38cm}( A\ot B\ot  \rho_{H\ot
H}^1)\co \Gamma_{A\ot B}\co (A\ot B\ot i_{A \bullet B})  $
\item[ ]$= (\mu_{A\ot B}\ot \rho_{H\ot H}^1 )\co (A\ot B\ot (((\rho_{A\otimes B}^1\co i_{A\bullet B})\ot H)\co
\rho_{A\bullet B})) $
\item[ ]$= (\mu_{A\ot B}\ot \rho_{H\ot H}^1 )\co (A\ot B\ot ((i_{A\bullet B}\ot H\ot H)\co
(\rho_{A\bullet B}\ot H)\co \rho_{A\bullet B}))  $
\item[ ]$= (\mu_{A\ot B}\ot H\ot (c_{H,H}\co \delta_{H}))\co (A\ot B\ot
((i_{A\bullet B}\ot \delta_{H})\co \rho_{A\bullet B}))  $
\item[ ]$= (\mu_{A\ot B}\ot H\ot \delta_{H})\co (A\ot B\ot ((i_{A\bullet B}\ot H\ot H)\co
(\rho_{A\bullet B}\ot H)\co \rho_{A\bullet B}))  $
\item[ ]$=  ( A\ot B\ot  \rho_{H\ot
H}^2)\co \Gamma_{A\ot B}\co (A\ot B\ot i_{A \bullet B})  $
\end{itemize}
where the first equality follows by the naturality of $c$ and the
properties of $\rho_{A\bullet B}$, the second and the third ones by
the comodule structure of $A\bullet B$, the fourth one by the
cocommutativity of $H$ and the last one was obtained repeating the
same calculus with $\rho_{H\ot H}^2$.

As a consequence, there exists a unique morphism $h:A\ot B\ot
A\bullet B\rightarrow A\ot B\ot H$ such that
\begin{equation}
\label{ff-gal-1} (A\ot B\ot \delta_H)\co h=\Gamma_{A\ot B}\co (A\ot
B\ot i_{A \bullet B}).
\end{equation}

On the other hand, in an analogous way the morphism $\Gamma_{A\ot B}^{-1}\co (A\ot B\ot
\delta_{H}):A\ot B\ot H\rightarrow A\ot B\ot A\ot B$ factorizes through 
through the equalizer $A\ot B\ot i_{A\bullet B}$ because by (i) of
Lemma \ref{lemma-2}, the naturality and symmetry of $c$ and the cocommutativity of $H$ we have
\begin{itemize}
\item[ ]$\hspace{0.38cm}( A\ot B\ot  \rho_{A\ot
B}^1)\co \Gamma_{A\ot B}^{-1}\co (A\ot B\ot \delta_{H})  $
\item[ ]$= (A\ot c_{A,B}\ot c_{H,B})\co (A\ot A\ot c_{H,B}\ot B)\co (((A\ot \rho_{A})\co \gamma_{A}^{-1})
\ot \gamma_{B}^{-1})\co (A\ot c_{B,H}\ot H)\co (A\ot B\ot
\delta_{H}) $
\item[ ]$=(A\ot c_{A,B}\ot c_{H,B})\co (A\ot A\ot c_{H,B}\ot B)\co (((\gamma_{A}^{-1}\ot
H)\co (A\ot \delta_{H})) \ot \gamma_{B}^{-1})\co (A\ot c_{B,H}\ot
H)\co (A\ot B\ot \delta_{H})   $
\item[ ]$=(A\ot c_{A,B}\ot B\ot H)\co (\gamma_{A}^{-1}\ot ((\gamma_{B}^{-1}\ot
H)\co (B\ot \delta_{H})))\co (A\ot c_{B,H}\ot H)\co (A\ot B\ot
\delta_{H}) $
\item[ ]$= ( A\ot B\ot  \rho_{A\ot
B}^2)\co \Gamma_{A\ot B}^{-1}\co (A\ot B\ot \delta_{H})    $
\end{itemize}
Thus, let $g$ be the unique morphism such that
\begin{equation}
\label{ff-gal-2} (A\ot B\ot i_{A \bullet B})\co g=\Gamma_{A\ot
B}^{-1}\co (A\ot B\ot \delta_{H}).
\end{equation}
By (\ref{ff-gal-1}) and (\ref{ff-gal-2})
$$(A\ot B\ot \delta_{H})\co h\co g=\Gamma_{A\ot B}\co (A\ot
B\ot i_{A \bullet B})\co g=\Gamma_{A\ot B}\co \Gamma_{A\ot
B}^{-1}\co (A\ot B\ot \delta_{H})=A\ot B\ot \delta_{H},$$
$$(A\ot B\ot i_{A \bullet B})\co g\co h=\Gamma_{A\ot
B}^{-1}\co (A\ot B\ot \delta_{H})\co h=\Gamma_{A\ot
B}^{-1}\co\Gamma_{A\ot B}\co (A\ot B\ot i_{A \bullet B})=A\ot B\ot
i_{A \bullet B}$$ and then we obtain that $h$ is an isomorphism with
inverse $g$. As a consequence $A\bullet B$ is faithfully flat because $A$, $B$ and $H$ are faithfully flat.

The morphism $\Gamma_{A\ot B}^{-1}\co (i_{A\bullet B}\ot
\delta_{H}):A\bullet B\ot H\rightarrow A\ot B\ot A\ot B$ admits a
factorization $\alpha_{A,B}:A\bullet B\ot H\rightarrow A\ot B\ot
A\bullet B$ through the equalizer $A\ot B\ot i_{A\bullet B}$ because
as we saw in the previous lines $\Gamma_{A\ot B}^{-1}\co (A\ot B\ot \delta_{H})$ admits a
factorization through $A\ot B\ot i_{A\bullet B}$.

Now consider the equalizer diagram
$$
 \setlength{\unitlength}{3mm}
\begin{picture}(45,4)
\put(11,2){\vector(1,0){6}} \put(27,2.5){\vector(1,0){7}}
\put(27,1.5){\vector(1,0){7}} \put(6,2){\makebox(0,0){$A\bullet B\ot
A\bullet B$}} \put(22,2){\makebox(0,0){$A\ot B\ot A\bullet B$}}
\put(40,2){\makebox(0,0){$A\ot B\ot H\ot A\bullet B$}}
\put(14,3){\makebox(0,0){$\scriptstyle i_{A\bullet B}\ot A\bullet
B$}} \put(30,3.5){\makebox(0,0){$\scriptstyle \rho_{A\ot B}^1\ot
A\bullet B$}} \put(30,0.5){\makebox(0,0){$\scriptstyle
\rho_{A\ot B}^2\ot A\bullet B$}}
\end{picture}
$$
We have that
\begin{itemize}
\item[ ]$\hspace{0.38cm} (\rho_{A\ot B}^1\ot
i_{A\bullet B})\co \alpha_{A,B} $
\item[ ]$=(\rho_{A\ot B}^{1}\ot A\ot B)\co (A\ot c_{A,B}\ot B)\co (\gamma_{A}^{-1}\ot \gamma_{B}^{-1})\co
(A\ot c_{B,H}\ot H)\co ( i_{A\bullet B}\ot \delta_{H})   $
\item[ ]$= (A\ot ((B\ot c_{A,H})\co (c_{A,B}\ot H)\co (A\ot c_{H,B}))\ot B)  $
\item[ ]$\hspace{0.38cm}\co (((\rho_{A\ot H}^{1}\co
\gamma_{A}^{-1})\ot \gamma_{B}^{-1})) \co (A\ot c_{B,H}\ot H)\co (
i_{A\bullet B}\ot \delta_{H})$
\item[ ]$= (A\ot ((B\ot c_{A,H})\co (c_{A,B}\ot H)\co (A\ot c_{H,B}))\ot B) $
\item[ ]$\hspace{0.38cm}\co (((\gamma_{A}^{-1}\ot H)\co (A\ot c_{H,H})\co (A\ot \mu_{H}\ot H)\co
(\rho_{A}\ot ((\lambda_{H}\ot H)\co \delta_{H})))\ot
\gamma_{B}^{-1})\co (A\ot c_{B,H}\ot H)\co ( i_{A\bullet B}\ot
\delta_{H})$
\item[ ]$= (A\ot ((B\ot c_{A,H})\co (c_{A,B}\ot H)\co (A\ot c_{H,B}))\ot B)\co
(\gamma_{A}^{-1}\ot H\ot  \gamma_{B}^{-1})  $
\item[ ]$\hspace{0.38cm}\co (A\ot ((c_{H,H}\ot B)\co (\mu_{H}\ot
c_{B,H})\co (H\ot c_{B,H}\ot H)\co (c_{B,H}\ot ((\lambda_{H}\ot
H)\co \delta_{H})))\ot H)$
\item[ ]$\hspace{0.38cm} \co((\rho_{A\ot B}^1\co i_{A\bullet B})\ot
\delta_{H})$
\item[ ]$= (A\ot ((B\ot c_{A,H})\co (c_{A,B}\ot H)\co (A\ot c_{H,B}))\ot B)\co
(\gamma_{A}^{-1}\ot H\ot B\ot B)   $
\item[ ]$\hspace{0.38cm} \co (A\ot ((c_{H,H}\ot \gamma_{B}^{-1})\co (H\ot c_{B,H}\ot H)\co (c_{B,H}\ot H\ot H)\co (B\ot
((\mu_{H}\co (H\ot \lambda_{H}))\ot H)\ot H)))   $
\item[ ]$\hspace{0.38cm}\co ((\rho_{A\ot B}^2\co i_{A\bullet B})\ot
((\delta_{H}\ot H)\co \delta_{H}))$
\item[ ]$= (A\ot ((c_{H,B}\ot A)\co (H\ot c_{A,B})\co (c_{A,H}\ot B))\ot B)\co
(\gamma_{A}^{-1}\ot H\ot \gamma_{B}^{-1})   $
\item[ ]$\hspace{0.38cm}\co (A\ot ((H\ot c_{B,H}\ot H)\co (c_{B,H}\ot (c_{H,H}\co c_{H,H}))\co (B\ot c_{H,H}\ot H)))
\co (A\ot B\ot \mu_{H}\ot (c_{H,H}\co \delta_{H}))   $
\item[ ]$\hspace{0.38cm} \co ((\rho_{A\ot B}^2\co i_{A\bullet B})\ot
((\lambda_{H}\ot H)\co \delta_{H}))$
\item[ ]$= (A\ot ((B \ot c_{A,H}\ot B)\co (c_{A,B}\ot c_{B,H}))\co (\gamma_{A}^{-1}\ot
\gamma_{B}^{-1}\ot H)\co (A\ot c_{B,H}\ot c_{H,H})$
\item[ ]$\hspace{0.38cm}\co (A\ot B\ot
(c_{H,H}\co (\mu_{H}\ot H))\ot H)  \co ( ((A\ot \rho_{B})\co
i_{A\bullet B})\ot ((\lambda_{H}\ot (c_{H,H}\co\delta_{H}))\co \delta_{H}))$
\item[ ]$=(A\ot ((B \ot c_{A,H}\ot B)\co (c_{A,B}\ot c_{B,H})) )  $
\item[ ]$\hspace{0.38cm}  \co (\gamma_{A}^{-1}\ot ((\gamma_{B}^{-1}\ot H)\co (B\ot c_{H,H})\co (B\ot
\mu_{H}\ot H)\co (\rho_{B}\ot ((\lambda_{H}\ot H)\co \delta_{H})) ))
$
\item[ ]$\hspace{0.38cm} \co (A\ot c_{B,H}\ot H)\co (i_{A\bullet B}\ot
(c_{H,H}\co \delta_{H}))$
\item[ ]$=(A\ot  ((B\ot c_{A,H})\co (c_{A,B}\ot H)\co (A\ot
\rho_{B}))\ot B)\co (\gamma_{A}^{-1}\ot  \gamma_{B}^{-1})\co
(A\ot c_{B,H}\ot H)\co (i_{A\bullet B}\ot \delta_{H})$
\item[ ]$ =(\rho_{A\ot B}^2\ot
i_{A\bullet B})\co \alpha_{A,B} $
\end{itemize}
where the first equality follows by the definition, the second, the
fourth and the fifth ones  by the naturality and symmetry of $c$,
the third and the nineth ones by (ii) of Lemma \ref{lemma-2},  the
sixth one by the cocommutativity of $H$ and, finally,  the eighth
and the tenth  ones by the naturality of $c$.

Then, there exists a unique morphism $\beta_{A,B}:A\bullet B\ot
H\rightarrow A\bullet B\ot A\bullet B$ such that
\begin{equation}
\label{inv-can} (i_{A\bullet B}\ot A\bullet B)\co
\beta_{A,B}=\alpha_{A,B}
\end{equation}
and then
\begin{equation}
\label{inv-can-1} (i_{A\bullet B}\ot i_{A\bullet B})\co
\beta_{A,B}=\Gamma_{A\ot B}^{-1}\co (i_{A\bullet B}\ot \delta_{H})
\end{equation}

The morphism $\beta_{A,B}$ satisfies
\begin{itemize}
\item[ ]$\hspace{0.38cm} (i_{A\bullet B}\ot H)\co \gamma_{A\bullet B}\co \beta_{A,B} $
\item[ ]$= (\mu_{A\ot B}\ot H)\co (i_{A\bullet B} \ot (  (i_{A\bullet B}\ot H)\co
\rho_{A\bullet B}))\co \beta_{A,B}$
\item[ ]$= (\mu_{A\ot B}\ot H)\co (i_{A\bullet B} \ot (  (A\ot \rho_{B})\co i_{A\bullet B} ))\co \beta_{A,B}   $
\item[ ]$=(\mu_{A}\ot \gamma_{B})\co (A\ot c_{B,A}\ot B)\co \Gamma_{A\ot B}^{-1}\co
(i_{A\bullet B}\ot \delta_{H})   $
\item[ ]$= (((A\ot \varepsilon_{H})\co \gamma_{A}\co
\gamma_{A}^{-1})\ot B\ot H)\co (A\ot c_{B,H}\ot H)\co (i_{A\bullet
B}\ot \delta_{H})    $
\item[ ]$=i_{A\bullet B}\ot H$
\end{itemize}
and by the cocommutativity of $H$ we have
\begin{itemize}
\item[ ]$\hspace{0.38cm} (i_{A\bullet B}\ot i_{A\bullet B})\co \beta_{A,B}\co \gamma_{A\bullet B}  $
\item[ ]$= \Gamma_{A\ot B}^{-1}\co (i_{A\bullet B}\ot \delta_{H})\co (\mu_{A\bullet B}\ot H)\co ( A\bullet B
\ot \rho_{A\bullet B})   $
\item[ ]$=  \Gamma_{A\ot B}^{-1}\co (\mu_{A\ot B}\ot \delta_{H})\co
(i_{A\bullet B}\ot ((A\ot \rho_{B})\co i_{A\bullet B}))   $
\item[ ]$= \Gamma_{A\ot B}^{-1}\co (A\ot B\ot c_{H,H})\co (\mu_{A}\ot \gamma_{B}\ot H)\co (A\ot c_{B,A}\ot \rho_{B})\co
(i_{A\bullet B}\ot i_{A\bullet B})      $
\item[ ]$= (A\ot c_{A,B}\ot B)\co (\gamma_{A}^{-1}\ot B\ot B)\co (A\ot c_{B,H}\ot B)\co (\mu_{A}\ot B\ot (c_{B,H}\co\rho_{B})) \co
(A\ot c_{B,A}\ot B)\co  (i_{A\bullet B}\ot i_{A\bullet B})   $
\item[ ]$=  (A\ot c_{A,B}\ot B)\co (\gamma_{A}^{-1}\ot B\ot B)\co
(\mu_{A}\ot c_{B,H}\ot B)\co 
(A\ot c_{B,A}\ot H\ot  B) \co  (i_{A\bullet B}\ot ((\rho_{A}\ot B)\co
i_{A\bullet B})) $
\item[ ]$=(((A\ot c_{A,B})\co ((\gamma_{A}^{-1}\co\gamma_{A})\ot
B))\ot B)\co (A\ot c_{B,A}\ot B)\co (i_{A\bullet B}\ot i_{A\bullet
B}) $
\item[ ]$=i_{A\bullet B}\ot i_{A\bullet
B}     $
\end{itemize}
Taking into account that $H$ is flat and that $A\bullet B$ is faithfully flat
we obtain that $\beta_{A,B}$ is the inverse of the canonical
morphism $\gamma_{A\bullet B}$.

Now we assume that ${\Bbb A}$ and ${\Bbb B}$ are strong Galois $H$-objects. To prove that ${\Bbb A}\bullet {\Bbb B}$ is a strong Galois $H$-object we only need to show that $f_{A\bullet
B}:H\rightarrow (A\bullet B)^e$ is a morphism of unital magmas. If
$f_{A}$ and $f_{B}$ are morphisms of unital magmas, by the
properties of $i_{A\bullet B}$ and the naturality of $c$ we have
$$(i_{A\bullet B}\ot i_{A\bullet B})\co f_{A\bullet
B}\co  \eta_{H}=(A\ot c_{A,B}\ot B)\co ((f_{A}\co \eta_{H})\ot
 (f_{B}\co \eta_H))$$
 $$=\eta_{A}\ot \eta_{B}\ot \eta_{A}\ot \eta_{B}=(i_{A\bullet B}\ot i_{A\bullet B})\co
\eta_{(A\bullet B)^e} $$ and
$$(i_{A\bullet B}\ot i_{A\bullet B})\co \mu_{(A\bullet B)^e}\co (f_{A\bullet
B}\ot f_{A\bullet B})=(A\ot c_{A,B}\ot B)\co ((\mu_{A^e}\co
(f_{A}\ot f_{A}))\ot (\mu_{B^e}\co (f_{B}\ot f_{B}))\co \delta_{H\ot
H}$$
$$= (A\ot c_{A,B}\ot B)\co (f_{A}\ot f_{B})\co \delta_{H}\co \mu_{H}= (i_{A\bullet B}\ot i_{A\bullet B})\co
f_{A\bullet B}\co \mu_{H}.$$

Therefore, $f_{A\bullet B}\co  \eta_{H}=\eta_{(A\bullet B)^e}$ and
$\mu_{(A\bullet B)^e}\co (f_{A\bullet B}\ot f_{A\bullet
B})=f_{A\bullet B}\co \mu_{H}$.
\end{proof}

\begin{prop}
\label{op-Gal} Let $H$ be a cocommutative  Hopf quasigroup and
${\Bbb A}$  a Galois $H$-object. Then the right
$H$-comodule magma $\overline{{\Bbb A}}$ defined in Proposition \ref
{op-magma} is a Galois $H$-object. Moreover, if ${\Bbb A}$ is
strong so is $\overline{{\Bbb A}}$.

\end{prop}

\begin{proof} To prove that $\overline{{\Bbb A}}$ is a Galois $H$-object we only need to show that
$\gamma_{\overline{A}}$ is an isomorphism.  We begin by proving the
following identity:
\begin{equation}
\label{aux-op} (A\ot (\mu_{H}\co c_{H,H}\co (\lambda_{H}\ot H)))\co
(\rho_{A}\ot H)\co \gamma_{A}=\gamma_{\overline{A}}\co c_{A,A}.
\end{equation}
Indeed:
\begin{itemize}
\item[ ]$\hspace{0.38cm}(A\ot (\mu_{H}\co c_{H,H}\co (\lambda_{H}\ot H)))\co
(\rho_{A}\ot H)\co \gamma_{A}  $
\item[ ]$=(\mu_{A}\ot (\mu_{H}\co (H\ot \mu_{H})\co (c_{H,H}\ot H)\co (H\ot c_{H,H})\co (c_{H,H}\ot H)\co
(\lambda_{H}\ot \lambda_{H}\ot H)\co (H\ot \delta_{H})))  $
\item[ ]$\hspace{0.38cm} \co (A\ot c_{H,A}\ot H)\co (\rho_{A}\ot
\rho_{A})$
\item[ ]$= (\mu_{A}\ot (\mu_{H} \co (\lambda_{H}\ot \mu_{H})\co (\delta_{H}\ot H)\co c_{H,H}\co
(\lambda_{H}\ot H)))\co (A\ot c_{H,A}\ot H)\co (\rho_{A}\ot
\rho_{A})$
\item[ ]$= (\mu_{A}\ot H)\co (A\ot c_{H,A})\co (\rho_{\overline{A}}\ot A)  $
\item[ ]$=\gamma_{\overline{A}}\co c_{A,A}   $
\end{itemize}
where the first equality follows by (b2) of Definition
\ref{H-comodule-magma}, (\ref{anti}) and the naturality of $c$, the
second one by the cocommutativity of $H$ and the naturality of $c$,
the third one by (a2-1) of Definition \ref{Hopf quasigroup} and the last one by the
symmetry and naturality of $c$.

Define the morphism $\gamma_{\overline{A}}^{\prime}:A\ot
H\rightarrow A\ot A$ by
\begin{equation}
\label{gamma-op-inv} \gamma_{\overline{A}}^{\prime}=c_{A,A}\co
\gamma_{A}^{-1}\co (A\ot (\mu_{H}\co c_{H,H}))\co (\rho_{A}\ot H).
\end{equation}
Then, by (\ref{aux-op}), the naturality of $c$, the cocommutativity
of $H$ and (a2-2) of Definition \ref{Hopf quasigroup}, we have the
following:
\begin{itemize}
\item[ ]$\hspace{0.38cm} \gamma_{\overline{A}}\co \gamma_{\overline{A}}^{\prime}  $
\item[ ]$=\gamma_{\overline{A}}\co c_{A,A}\co \gamma_{A}^{-1}\co (A\ot
(\mu_{H}\co c_{H,H}))\co (\rho_{A}\ot H)  $
\item[ ]$=(A\ot (\mu_{H}\co c_{H,H}\co (\lambda_{H}\ot H)))\co(\rho_{A}\ot (\mu_{H}\co c_{H,H}))\co (\rho_{A}\ot H) $
\item[ ]$=(A\ot (\mu_{H}\co (\mu_{H}\ot H)\co (H\ot c_{H,H})\co (c_{H,H}\ot H)\co
(\lambda_{H}\ot c_{H,H})\co (\delta_{H}\ot H)))\co (\rho_{A}\ot H) $
\item[ ]$= (A\ot ( (\mu_H\circ (\mu_H\ot H)\circ (H\ot \lambda_H\ot H)
\circ (H\ot \delta_H))\co c_{H,H}))\co (\rho_{A}\ot H) $
\item[ ]$= id_{A\ot H}.  $
\end{itemize}
Moreover, by a similar reasoning but using (a2-1) of Definition
\ref{Hopf quasigroup} instead of (a2-2) we obtain
\begin{itemize}
\item[ ]$\hspace{0.38cm} \gamma_{\overline{A}}^{\prime}\co \gamma_{\overline{A}}  $
\item[ ]$= c_{A,A}\co \gamma_{A}^{-1}\co ((\mu_{A}\co c_{A,A})\ot (\mu_H\circ
(H\ot \mu_H)\circ (H\ot \lambda_H\ot H)\circ (\delta_H\ot
H)\co c_{H,H}))\co (A\ot c_{H,A}\ot H)\co (\rho_{A}\ot \rho_{A}) $
\item[ ]$=c_{A,A}\co \gamma_{A}^{-1}\co\gamma_{A}\co c_{A,A} $
\item[ ]$= id_{A\ot A}.  $
\end{itemize}
Therefore, $\gamma_{\overline{A}}$ is an isomorphism and
$\overline{{\Bbb A}}$ a Galois $H$-object.

Finally, it is easy to show that $f_{\overline{A}}=c_{A,A}\co
f_{A}$. Then, if $f_{A}$ is a morphism of unital magmas,
so is $f_{\overline{A}}$. Thus if ${\Bbb A}$ is strong,
$\overline{{\Bbb A}}$ is strong.

\end{proof}

\begin{prop}
\label{op-Gal-1} Let $H$ be a cocommutative flat Hopf quasigroup and
${\Bbb A}$  a Galois $H$-object. Then ${\Bbb A}\bullet
\overline{{\Bbb A}}$ is isomorphic to ${\Bbb H}$ as right
$H$-comodules. Moreover, if  ${\Bbb A}$ is strong, the previous
isomorphism is a morphism of right $H$-comodule magmas.

\end{prop}

\begin{proof} First note that, by Proposition
\ref{trivial-coinvaiants}, we know that (\ref{coinvariants=K}) is an
equalizer diagram and then so is
$$
\setlength{\unitlength}{3mm}
\begin{picture}(30,4)
\put(1,2){\vector(1,0){5}} \put(11,2.5){\vector(1,0){10}}
\put(11,1.5){\vector(1,0){10}} \put(0,2){\makebox(0,0){$H$}}
\put(25,2){\makebox(0,0){$A\ot H\ot H$}}
\put(4,3){\makebox(0,0){$\eta_{A}\ot H$}}
\put(16,3.5){\makebox(0,0){$\rho_{A}\ot H$}}
\put(16,0.5){\makebox(0,0){$A\ot \eta_{H}\ot H$}}
\put(8.5,2){\makebox(0,0){$A\ot H$}}
\end{picture}
$$
because $H$ is flat. For the morphism $
\gamma_{\overline{A}}\co i_{A\bullet \overline{A}}:A\bullet
\overline{A}:\rightarrow A\ot H$ we have the following:

\begin{itemize}
\item[ ]$\hspace{0.38cm} (\rho_{A}\ot H)\co \gamma_{\overline{A}}\co i_{A\bullet \overline{A}}  $
\item[ ]$= (\mu_{\overline{A}\ot \overline{H}}\ot \lambda_{H})\co  (\rho_{A}\ot
((A\ot \delta_{H})\co \rho_{A}))\co i_{A\bullet \overline{A}} $
\item[ ]$=(\mu_{\overline{A}\ot \overline{H}}\ot \lambda_{H})\co
(A\ot H\ot ((A\ot \delta_{H})\co \rho_{A}))\co (A\ot (c_{A,H}\co
(A\ot \lambda_{H})\co \rho_{A}))\co i_{A\bullet \overline{A}}  $
\item[ ]$= (\mu_{\overline{A}}\ot \mu_{\overline{H}}\ot \lambda_{H})\co (A\ot
((A\ot H\ot  \delta_{H})\co (A\ot c_{H,H})\co
(A\ot ((H\ot \lambda_{H})\co \delta_{H}))))\co (A\ot \rho_{A}) \co
i_{A\bullet \overline{A}}$
\item[ ]$=( \mu_{\overline{A}}\ot ((\mu_{H}\co ((\lambda_{H}\ot H)\co \delta_{H})\ot
\lambda_{H})\co \delta_{H}))\co (A\ot \rho_{A}) \co i_{A\bullet
\overline{A}} $
\item[ ]$= (A\ot \eta_{H}\ot H)\co \gamma_{\overline{A}}\co i_{A\bullet \overline{A}}$
\end{itemize}
where the first equality follows by the naturality of $c$ and (b2)
of Definition \ref{H-comodule-magma}, the second one because $\rho_{A\ot
\overline{A}}^{1}\co i_{A\bullet \overline{A}}=\rho_{A\ot
\overline{A}}^{2}\co i_{A\bullet \overline{A}}$, the third one relies on 
the symmetry and the naturality of $c$, the fourth one follows by (\ref{anti})
and the last one by (\ref{antipode}).

Therefore, there exists an unique morphism $h_{A}:A\bullet
\overline{A}\rightarrow H$ such that
\begin{equation}
\label{h-iso} (\eta_{A}\ot H)\co h_{A}=\gamma_{\overline{A}}\co
i_{A\bullet \overline{A}}.
\end{equation}
The morphism $h_{A}$ is a right comodule morphism because by the
cocommutativity of $H$, (\ref{anti}) and the comodule properties of
$A$, we have
\begin{itemize}
\item[ ]$\hspace{0.38cm} \eta_{A}\ot(( h\ot H)\co \rho_{A\bullet \overline{A}})  $
\item[ ]$= ((\gamma_{\overline{A}}\co
i_{A\bullet \overline{A}})\ot H)\co \rho_{A\bullet \overline{A}} $
\item[ ]$= ((\mu_{A}\co c_{A,A})\ot ((\lambda_{H}\ot \lambda_{H})\co \delta_{H}))\co
(A\ot \rho_{A})\co i_{A\bullet \overline{A}}  $
\item[ ]$= ((\mu_{A}\co c_{A,A})\ot ((\lambda_{H}\ot \lambda_{H})\co c_{H,H}\co \delta_{H}))\co
(A\ot \rho_{A})\co i_{A\bullet \overline{A}}$
\item[ ]$=((\mu_{A}\co c_{A,A})\ot ( \delta_{H}\co \lambda_{H}))\co
(A\ot \rho_{A})\co i_{A\bullet \overline{A}} $
\item[ ]$= (A\ot \delta_{H})\co \gamma_{\overline{A}}\co
i_{A\bullet \overline{A}}$
\item[ ]$=\eta_{A}\ot (\delta_{H}\co h)$.
\end{itemize}
and using that $\eta_{A}\ot H\ot H$ is an equalizer morphism we obtain
$(h\ot H)\co \rho_{A\bullet \overline{A}}=\delta_{H}\co h.$

On the other hand, for $f_{\overline{A}}: H\rightarrow A\ot A$ we
have the following
\begin{itemize}
\item[ ]$\hspace{0.38cm}  (\gamma_{\overline{A}}\ot H)\co \rho_{A\ot \overline{A}}^{1}\co
f_{\overline{A}}$
\item[ ]$=(\mu_{A}\ot \lambda_{H}\ot H)\co (c_{A,A}\ot H\ot H)\co (A\ot \rho_{A}\ot H)\co (A\ot c_{H,A})\co
(\rho_{A}\ot A)\co c_{A,A}\co f_{A}   $
\item[ ]$=  (\mu_{A}\ot \lambda_{H}\ot H)\co  (A\ot c_{H,A}\ot H)\co (\rho_{A}\ot\rho_{A})\co
f_{A} $
\item[ ]$=  (\mu_{A}\ot ((\lambda_{H}\ot H)\co \gamma_{H}^{-1}\co \gamma_{H}))\co  (A\ot c_{H,A}\ot H)\co (\rho_{A}\ot\rho_{A})\co
f_{A} $
\item[ ]$=(A\ot (\lambda_{H}\co \mu_{H})\ot H)\co (\rho_{A}\ot ((\lambda_{H}\ot H)\co \delta_{H}))\co \gamma_{A}
\co f_{A}   $
\item[ ]$=\eta_{A}\ot(((\lambda_{H}\co \lambda_{H})\ot H)\co \delta_{H})$
\item[ ]$=\eta_{A}\ot \delta_{H}$
\end{itemize}
where the first equality follows because $f_{\overline{A}}=c_{A,A}\co f_{A}$, the second one
 by the symmetry and the naturality of $c$. In  the third one
we used that ${\Bbb H}$ is a Galois $H$-object and the fourth and the sixth ones are 
a consequence of (b1) of Definition \ref{H-comodule-magma}. Finally,
in the fifth one we applied that ${\Bbb A}$ is a Galois $H$-object,
and the last one relies on the cocommutativity of $H$. Also
\begin{itemize}
\item[ ]$\hspace{0.38cm}  (\gamma_{\overline{A}}\ot H)\co \rho_{A\ot \overline{A}}^{2}\co
f_{\overline{A}} $
\item[ ]$= ((\mu_{A}\co c_{A,A})\ot ((\lambda_{H}\ot \lambda_{H})\co \delta_{H}))\co
(A\ot \rho_{A})\co c_{A,A}\co f_{A} $
\item[ ]$= (\mu_{A} \ot ((\lambda_{H}\ot \lambda_{H})\co \delta_{H}))\co (A\ot c_{H,A})
\co (\rho_{A}\ot A)\co f_{A} $
\item[ ]$= (\mu_{A} \ot ((\varepsilon_{H}\ot ((\lambda_{H}\ot \lambda_{H})\co \delta_{H}))\co
c_{H,H}\co ((\mu_H\circ(\mu_H\ot \lambda_H)\circ (H\ot \delta_H))\ot
H))) $
\item[ ]$\hspace{0.38cm}\co (A\ot c_{H,A}\ot \delta_{H})\co
(\rho_{A}\ot \rho_{A})\co f_{A}$
\item[ ]$=(A \ot ((\lambda_{H}\ot \lambda_{H})\co \delta_{H}))\co (A\ot \mu_{H})\co  ((\rho_{A}\co \eta_{A})\ot \lambda_{H})   $
\item[ ]$=\eta_{A}\ot \delta_{H},$
\end{itemize}
where the first equality follows by (b1) of Definition
\ref{H-comodule-magma} and the comodule properties of $A$, the
second one by the naturality of $c$, the third one by (a2-2) of
Definition \ref{Hopf quasigroup} and the counit properties, the
fourth one by (b2) of Definition \ref{H-comodule-magma} and in the
last one we used that ${\Bbb A}$ is a Galois $H$-object, (b1) of
Definition \ref{H-comodule-magma} and the cocommutativity of $H$.

Then, $\rho_{A\ot \overline{A}}^{1}\co f_{\overline{A}}=\rho_{A\ot
\overline{A}}^{2}\co f_{\overline{A}}$ and, as a consequence, there
exists a unique morphism $h_{A}^{\prime}:H\rightarrow A\bullet
\overline{A}$ such that
\begin{equation}
\label{h-iso-1} i_{A\bullet \overline{A}}\co
h_{A}^{\prime}=f_{\overline{A}}.
\end{equation}

Therefore, by (\ref{h-iso}) and (\ref{h-iso-1}) we have
$$i_{A\bullet \overline{A}}\co
h_{A}^{\prime}\co h_{A}=f_{\overline{A}}\co
h_{A}=\gamma_{\overline{A}}^{-1}\co(\eta_{A}\ot H)\co h_{A}=
\gamma_{\overline{A}}^{-1}\co\gamma_{\overline{A}}\co i_{A\bullet
\overline{A}}=i_{A\bullet \overline{A}}$$ and
$$(\eta_{A}\ot H)\co
h_{A}\co h_{A}^{\prime}=\gamma_{\overline{A}}\co i_{A\bullet
\overline{A}}\co h_{A}^{\prime}=\gamma_{\overline{A}}\co
f_{\overline{A}}=
\gamma_{\overline{A}}\co\gamma_{\overline{A}}^{-1}\co (\eta_{A}\ot
H)=(\eta_{A}\ot H).$$ Then, $h_{A}^{\prime}\co h_{A}=id_{A\bullet
\overline{A}}$ and $h_{A}\co h_{A}^{\prime}=id_{H}$ and $h$ is an
isomorphism.

Finally, assume that ${\Bbb A}$ is strong. By (\ref{h-iso-1}) and
the equality $f_{\overline{A}}= c_{A,A}\co f_{A}$ we obtain that
$h_{A}^{\prime}$ is a morphism of unital magmas. Then $h_{A}$ is a
morphism of unital magmas and  the proof is finished.
\end{proof}

\begin{rem}
\label{Gal-acla} {\rm Note that, in the Hopf algebra setting, for
any Galois $H$-object ${\Bbb A}$, the morphism $h_{A}$ obtained in
the previous proposition is a morphism of monoids because this
property can be deduced from the associativity of the product
defined in $A$. In the Hopf quasigroup world this proof does not
work because $A$  is  a magma. }
\end{rem}

\begin{teo}
\label{main-result}  Let $H$ be a cocommutative faithfully flat Hopf
quasigroup. The set of isomorphism classes of Galois $H$-objects is
a commutative monoid. Moreover, the set of isomorphism classes of
strong Galois $H$-objects is a commutative group.
\end{teo}

\begin{proof} Let $Gal_{{\mathcal C}}(H)$ be the set of isomorphism classes of Galois $H$-objects.
For a Galois $H$-object ${\Bbb A}$ we denote its class in
$Gal_{{\mathcal C}}(H)$ by $[{\Bbb A}]$. By  by Propositions
\ref{product-Gal} and \ref{morphisms}, the product
\begin{equation}
\label{gal-group} [{\Bbb A}].[{\Bbb B}]=[{\Bbb A}\bullet {\Bbb B}]
\end{equation}
is well-defined. By Propositions \ref{product-associative},
\ref{product-commutative} and \ref{product-unit} we obtain that
$Gal_{{\mathcal C}}(H)$ is a commutative monoid with unit $[{\Bbb
H}]$.

 If we denote by $Gal_{{\mathcal C}}^s(H)$ the set of isomorphism classes of
strong Galois $H$-objects, with the product defined in
(\ref{gal-group}) for Galois $H$-objects,  $Gal_{{\mathcal
C}}^s(H)$ is a commutative group because by (ii) of Proposition \ref{product-Gal} the product of strong Galois $H$-objects is a strong Galois $H$-object,  by Example \ref{H-strong}
we know that ${\Bbb H}$ is a strong Galois $H$-object and by
Propositions \ref{op-Gal} and \ref{op-Gal-1}, the inverse of $[{\Bbb
A}]$ in $Gal_{{\mathcal
C}}^s(H)$ is $[\overline{{\Bbb A}}]$.
\end{proof}

\begin{defin}
\label{normal basis}
{\rm  Let $H$ be a cocommutative faithfully flat Hopf quasigroup. If ${\Bbb A}$ is a (strong) Galois $H$-object, we will say that ${\Bbb A}$ has a normal basis if $(A,\rho_{A})$ is isomorphic to $(H,\delta_{H})$ as right $H$-comodules. We denote by $n_{A}$ the $H$-comodule isomorphism between $A$ and $H$.

Obviously, $N_{{\mathcal C}}(H)$, the set of isomorphism classes of Galois $H$-objects with normal basis,  is a submonoid of $Gal_{{\mathcal C}}(H)$ because ${\Bbb H}=(H,\delta_{H})$ is a Galois $H$-object with normal basis and if ${\Bbb A}$, ${\Bbb B}$  are Galois $H$-objects with normal basis and associated isomorphisms $n_{A}$, $n_{B}$ respectively, then ${\Bbb A}\bullet {\Bbb B}$ is a Galois $H$-object with normal basis and associated $H$-comodule isomorphism
$n_{A\bullet B}=r_{H}\co n_{A}\bullet n_{B}$  where $n_{A}\bullet n_{B}$ is defined as in Proposition \ref{morphisms} and $r_{H}$ is the isomorphism defined in Proposition \ref{product-unit}.  Moreover, for a strong Galois $H$-object with normal basis ${\Bbb A}$, with associated isomorphism $n_{A}$, we have that $\overline{{\Bbb A}}=(\overline{A},\rho_{\overline{A}})$ is also  a strong Galois $H$-object with normal basis, where $n_{\overline{A}}=\lambda_{H}\co n_{A}$, and then, if  we denote by $N_{{\mathcal C}}^s(H)$ the set of isomorphism classes of strong Galois $H$-objects with normal basis,  $N_{{\mathcal C}}^s(H)$ is a subgroup of  $Gal_{{\mathcal C}}^s(H)$.

 Note that, if $H$ is a Hopf algebra we have that $Gal_{{\mathcal C}}^s(H)=Gal_{{\mathcal C}}(H)$ and $N_{{\mathcal C}}^s(H)=N_{{\mathcal C}}(H)$. Therefore, in the associative setting we recover the classical  group of Galois $H$-objects.

}
\end{defin}

\begin{rem}
\label{k-theory}
{\rm
 In this remark we use some classical results of algebraic $K$-theory (see \cite{Bass} for the details). Let $G({\mathcal C},H)$ and $G^s({\mathcal C},H)$ be the  categories of Galois $H$-objects  and strong Galois $H$-objects, respectively. Then, by Proposition \ref{monoidal-magmas} these categories are symmetric monoidal  and then they are categories with product.  The Grothendieck group of  $G({\mathcal C},H)$ is the abelian group generated by the isomorphisms classes of objects ${\Bbb A}$ of $G({\mathcal C},H)$ module the relations $[{\Bbb A}\bullet {\Bbb B}]=[{\Bbb A}].[{\Bbb B}]$. This group will be denoted by $K_{0}G({\mathcal C},H)$ and, by the general theory of Grothendieck groups, we know that for ${\Bbb A}$,  ${\Bbb B}$ in $G({\mathcal C},H)$,  $[{\Bbb A}]=[{\Bbb B}]$ in $K_{0}G({\mathcal C},H)$ if and only if there exists a ${\Bbb D}$ in $G({\mathcal C},H)$ such that ${\Bbb A}\bullet {\Bbb D}$ is isomorphic in $G({\mathcal C},H)$ to ${\Bbb D}\bullet {\Bbb D}$. The unit of  $K_{0}G({\mathcal C},H)$ is $[{\Bbb H}]$. In a similar way we can define $K_{0}G^s({\mathcal C},H)$, but in this case $K_{0}G^s({\mathcal C},H)=Gal_{{\mathcal C}}^s(H)$ because the set of isomorphism classes of objects of $G^s({\mathcal C},H)$ is a group.

The inclusion functor $i:G^s({\mathcal C},H)\rightarrow G({\mathcal C},H)$ is a product preserving functor  and then we have a group morphism $K_{0}i: Gal_{{\mathcal C}}^s(H)\rightarrow K_{0}G({\mathcal C},H)$. If $[{\Bbb A}]\in Ker(K_{0}i)$ we have that $[{\Bbb A}]=[{\Bbb H}]$ in $K_{0}G({\mathcal C},H)$. Then there exists ${\Bbb D}$ in $G({\mathcal C},H)$ such that ${\Bbb A}\bullet {\Bbb D}\cong {\Bbb H}\bullet {\Bbb D} \cong  {\Bbb D} $ in $G({\mathcal C},H)$. As a consequence
${\Bbb A}\bullet {\Bbb D}\bullet \overline{{\Bbb D}}\cong {\Bbb D}\bullet \overline{{\Bbb D}}$ in $G({\mathcal C},H)$. Then, By Proposition \ref{op-Gal-1},  ${\Bbb A}\cong {\Bbb H}$ as right $H$-comodules. Therefore ${\Bbb A}$ is a strong Galois $H$-object with normal basis and  $Ker(K_{0}i)$ is a subgroup of $N_{{\mathcal C}}^s(H)$.

 The full subcategory  ${\mathcal H}=\{{\Bbb H}\}$ of $G^s({\mathcal C}, H)$ is cofinal because, for all 
${\Bbb A}$ in $G^s({\mathcal C}, H)$, ${\Bbb A}\bullet \overline{{\Bbb A}}\cong {\Bbb H}$ as right $H$-comodule magmas. Therefore, the Whitehead group of $G^s({\mathcal C}, H)$ is isomorphic to the 
Whitehead group of ${\mathcal H}$. Therefore,
$$K_{1}G^s({\mathcal C}, H)\cong Aut_{G^s({\mathcal C}, H)}( {\Bbb H}).$$

The group $Aut_{G^s({\mathcal C}, H)}( {\Bbb H})$ admits a good explanation in terms of grouplike elements of a suitable Hopf quasigroup if $H$ is finite, that is, if there exists an object $H^{\ast}$ in ${\mathcal C}$ and an adjunction $H\ot -\dashv H^{\ast}\ot -$. For this adjunction we will denote with  $a_{H}:id_{{\mathcal C}}\rightarrow H^{\ast}\ot H\ot -$ and $b_{H}:H\ot H^{\ast}\ot -\rightarrow id_{{\mathcal C}}$ the unit and the counit respectively. The object $H^{\ast}$ will be called the dual of $H$.

A Hopf coquasigroup $D$   in ${\mathcal
C}$ is a  monoid $(D, \eta_D, \mu_D)$ and a counital comagma $(D,
\varepsilon_D, \delta_D)$ such that the following axioms hold:
\begin{itemize}
\item[(d1)] $\varepsilon_D$ and $\delta_D$ are  morphisms of monoids.

\item[(d2)] There exists  $\lambda_{D}:D\rightarrow D$
in ${\mathcal C}$ (called the antipode of $D$) such that:

\begin{itemize}
\item[(d2-1)] $(\mu_D\ot D)\circ (\lambda_D\ot \delta_D)\circ \delta_D=
\eta_D\ot D= (\mu_D\ot D)\circ (D\ot ((\lambda_{D}\ot D)\co\delta_D))\circ \delta_D.$

\item[(d2-2)] $(D\ot \mu_D)\circ (\delta_{D}\ot \lambda_D)\circ \delta_D=
D\ot  \eta_D=(D\ot \mu_D)\circ (((D\ot \lambda_D)\co \delta_{D})\ot D)\circ 
\delta_D.$
\end{itemize}
\end{itemize}

As in the case of quasigroups, the antipode  is unique,
antimultiplicative, anticomultiplicative, leaves the unit and the
counit invariable and satisfies (\ref{antipode}).

If $D$ is a Hopf coquasigroup we define $G(D)$ as the set of morphisms  $h:K\rightarrow D$ such that 
$\delta_{D}\co h=h\ot h$ and $\varepsilon_{D}\co h=id_{K}$. If $D$ is commutative, $G(D)$  with the  convolution $h\ast g=\mu_{D}\co (h\ot g)$ is a  commutative group, called the group of grouplike morphisms of $D$. Note that the unit element of $G(D)$ is $\eta_{D}$ and the inverse of 
$h\in G(D)$  is $h^{-1}=\lambda_{D}\co h$. 

It is  easy to show that, if $H$ is a finite cocommutative Hopf quasigroup,  its dual $H^{\ast}$ is a commutative finite Hopf coquasigroup where:
$$\eta_{H^{\ast}}= 
(H^{\ast}\otimes \varepsilon_{H})\circ a_{H}, \;\; \mu_{H^{\ast}}= 
(H^{\ast}\otimes b_{H})\circ (H^{\ast}\otimes H\otimes b_{H}\otimes H^{\ast}) 
\circ (H^{\ast}\otimes \delta_{H}\otimes H^{\ast} 
\otimes H^{\ast})\circ (a_{H}\otimes H^{\ast}\otimes H^{\ast})), $$
$$\varepsilon_{H^{\ast}}=b_{H}\circ(\eta_{H}\otimes 
H^{\ast}),\;\; \delta_{H^{\ast}}= (H^{\ast}\otimes H^{\ast}\otimes (b_{H}\circ 
(\mu_{H}\otimes H^{\ast})))\circ (H^{\ast}\otimes 
a_{H}\otimes H\otimes H^{\ast})\circ (a_{H}\otimes H^{\ast}))$$ 
and the antipode is $(H^{\ast}\otimes b_{H})\circ 
(H^{\ast}\otimes \lambda_{H} \otimes H^{\ast})\circ (a_{H}\otimes H^{\ast})$. 

The groups $G(H^{\ast})$ and $Aut_{G^s({\mathcal C}, H)}( {\Bbb H})$ are isomorphic. The proof is equal to the  one given in Proposition 3.7 of \cite{Galois}. If $\alpha \in Aut_{G^s({\mathcal C}, H)}( {\Bbb H})$, the morphism $z_{\alpha}=(H^{\ast}\otimes (\varepsilon_{H}\circ \alpha))\circ a_{H}$ is in $G({\bf H}^{\ast})$. Then, we define the map $Aut_{G^s({\mathcal C}, H)}( {\Bbb H})\rightarrow G({\bf H}^{\ast})$ by $z(\alpha)=z_{\alpha}$. On the other hand, if $h\in G({\bf H}^{\ast})$, then $x_{h}=(H\otimes b_{H})\circ (\delta_{H}\otimes h):H\rightarrow H$  is a morphism of Galois 
${\bf H}$-objects and then, by Remark \ref{isomorphism}, it is an isomorphism, that is $x_{h}\in 
Aut_{G^s({\mathcal C}, H)}( {\Bbb H})$. The map  $x:G({\bf H}^{\ast})\rightarrow 
Aut_{G^s({\mathcal C}, H)}( {\Bbb H})$ defined by $x(h)=x_{h}$ is the inverse of $z$.
Therefore, 
$$K_{1}G^s({\mathcal C}, H)\cong G({\bf H}^{\ast}).$$
Finally, $N^s({\mathcal C}, H)$  is the subcategory of $G^s({\mathcal C}, H)$ whose objects are the strong Galois $H$-objects with normal basis, note that ${\mathcal H}=\{{\Bbb H}\}$ it is also cofinal in  $N^s({\mathcal C}, H)$ and then 
$$K_{1}N^s({\mathcal C}, H)\cong G({\bf H}^{\ast}).$$

}
\end{rem}

\section{Invertible comodules with geometric normal basis}

This section is devoted to  study the connections between Galois $H$-objects and invertible comodules with geometric normal basis. First of all, we introduce the notion of invertible comodule with geometric normal basis which is a generalization to the non associative setting of the one defined by Caenepeel in \cite{CAE}.

\begin{defin}
\label{geo-nor-bas}
{\rm  Let $H$ be a cocommutative faithfully flat Hopf quasigroup. A right $H$-comodule $\mathbf{M}=(M,\rho_{M})$ is called invertible with geometric normal basis if there exists a faithfully flat unital magma $S$ and an isomorphism $h_{M}:S\ot M\rightarrow S\ot H$ of right $H$-comodules such that $h_{M}$ is almost lineal, that is
\begin{equation}
\label{almost}
h_{M}=(\mu_{S}\ot H)\co (S\ot (h_{M}\co (\eta_{S}\ot M)))
\end{equation}
A morphism between two invertible right $H$-comodules with normal basis is a morphism of right $H$-comodules.

Note that, if $S$ is a monoid, $h_{M}$ is a morphism of left
$S$-modules, for $\varphi_{S\ot M}=\mu_{S}\ot M$ and $\varphi_{S\ot
H}=\mu_{S}\ot H$, if and only if (\ref{almost}) holds. Then in the
Hopf algebra setting this definition is the one introduced by
Caenepeel in \cite{CAE}. }
\end{defin}

\begin{ej}
{\rm Let $H$ be a cocommutative faithfully flat Hopf quasigroup and let ${\Bbb A}=(A,\rho_{A})$ be a Galois $H$-object. Then
$\mathbf{A}=(A,\rho_{A})$ is  an invertible right $H$-comodule with geometric normal basis because $h_{A}=\gamma_{A}$ is an isomorphism of right $H$-comodules and trivially  $\gamma_{A}$ is almost lineal. In particular, $\mathbf{H}=(H,\delta_{H})$ is an example of invertible right $H$-comodule with geometric normal basis.
}
\end{ej}

\begin{prop}
\label{product-geo-nor-bas}
Let $H$ be a cocommutative faithfully flat Hopf quasigroup  and $\mathbf{M}$, $\mathbf{N}$ be invertible right $H$-comodules with geometric normal basis. Then the right  $H$-comodule $\mathbf{M}\bullet \mathbf{ N}=(M\bullet N, \rho_{M\bullet N})$,  where  $M\bullet N$ and $\rho_{M\bullet N}$ are defined as in Proposition \ref{product}, is a right $H$-comodule with geometric normal basis.
\end{prop}

\begin{proof} Let $S$, $R$ and $h_{M}$, $h_{N}$ be the faithfully flat unital magmas and the isomorphisms of right $H$-comodules associated to $\mathbf{M}$ and $\mathbf{N}$ respectively.  Then $T=S\ot R$ is faithfully flat.  On the other hand,
$$
 \setlength{\unitlength}{3mm}
\begin{picture}(45,4)
\put(11,2){\vector(1,0){6}} \put(27,2.5){\vector(1,0){7}}
\put(27,1.5){\vector(1,0){7}} \put(6,2){\makebox(0,0){$T\ot
M\bullet N$}} \put(22,2){\makebox(0,0){$T\ot M\ot N$}}
\put(40,2){\makebox(0,0){$T\ot M\ot N \otimes H$}}
\put(14,3){\makebox(0,0){$\scriptstyle T\ot i_{M \bullet N}$}}
\put(30,3.5){\makebox(0,0){$\scriptstyle T\ot  \rho_{M\ot
N}^1$}} \put(30,0.5){\makebox(0,0){$\scriptstyle T\ot
\rho_{M\ot N}^2$}}
\end{picture}
$$
and
$$
 \setlength{\unitlength}{3mm}
\begin{picture}(45,4)
\put(11,2){\vector(1,0){6}} \put(27,2.5){\vector(1,0){7}}
\put(27,1.5){\vector(1,0){7}} \put(6,2){\makebox(0,0){$T\ot
H$}} \put(22,2){\makebox(0,0){$T\ot H\ot H$}}
\put(40,2){\makebox(0,0){$T\ot H\ot H \otimes H$}}
\put(14,3){\makebox(0,0){$\scriptstyle T\ot \delta_H$}}
\put(30,3.5){\makebox(0,0){$\scriptstyle T\ot  \rho_{H\ot
H}^1$}} \put(30,0.5){\makebox(0,0){$\scriptstyle T\ot
\rho_{H\ot H}^2$}}
\end{picture}
$$
are equalizer diagrams and for the morphism
$$g_{M\ot N}=(S\ot c_{R,H}\ot H)\co (h_{M}\ot  h_{N})\co (S\ot c_{R,M}\ot N):S\ot R\ot M\ot N\rightarrow S\ot R\ot H\ot H$$
we have that 
\begin{itemize}
\item[ ]$\hspace{0.38cm} (S\ot R\ot \rho_{H\ot H}^{1})\co g_{M\ot N}\co (S\ot R\ot i_{M\bullet N}) $
\item[ ]$= (S\ot c_{H,R}\ot c_{H,H})\co (S\ot H\ot c_{H,R}\ot H)\co
(((S\ot \delta_{H})\co h_{M})\ot h_{N})\co
(S\ot c_{R,M}\ot N)\co (S\ot R\ot i_{M\bullet N}) $
\item[ ]$=(S\ot c_{H,R}\ot c_{H,H})\co (S\ot H\ot c_{H,R}\ot H)\co
(((h_{M}\ot H)\co (S\ot \rho_{M}))\ot h_{N})\co (S\ot c_{R,M}\ot
N)\co (S\ot R\ot i_{M\bullet N}) $
\item[ ]$= (((S\ot c_{H,R}\ot H)\co (h_{M}\ot h_{N}))\ot H)\co (S\ot c_{R,M}\ot N\ot H)\co (S\ot R\ot 
((M\ot c_{H,N})\co (\rho_{M}\ot N)\co i_{M\bullet N}))  $
\item[ ]$= (((S\ot c_{H,R}\ot H)\co (h_{M}\ot h_{N}))\ot H)\co (S\ot c_{M,R}\ot
((M\ot \rho_{N})\co i_{M\bullet N}))  $
\item[ ]$=(S\ot R\ot \rho_{H\ot H}^{2})\co g_{M\ot N}\co (S\ot R\ot i_{M\bullet N}),  $
\end{itemize}
where the first and the third equalities follow by the naturality of
$c$, the second and the fifth ones by the comodule morphism
condition for $h_{M}$ and $h_{N}$ respectively and finally the
fourth one by the properties of $i_{M\bullet N}$.

Therefore, there exists a unique morphism $h_{M\bullet N}: T\ot
M\bullet N\rightarrow T\ot H$ such that
\begin{equation}
\label{product-geo-nor-bas-1} (T\ot \delta_{H})\co h_{M\bullet
N}=g_{M\ot N}\co (T\ot i_{M\bullet N}).
\end{equation}
Moreover, if we define the morphism
$$g_{M\ot N}^{\prime}=(S\ot c_{M,R}\ot N)\co (h_{M}^{-1} \ot h_{N}^{-1})\co
(S\ot c_{R,H}\ot H):S\ot R\ot H\ot H\rightarrow S\ot R\ot M\ot N$$
by the naturality of $c$, the comodule morphism condition for
$h_{M}^{-1}$and  $h_{N}^{-1}$ and the cocommutativity of $H$, the
following equalities hold
\begin{itemize}
\item[ ]$\hspace{0.38cm} (S\ot R\ot \rho_{M\ot N}^{1})\co g_{M\ot N}^{\prime}\co (S\ot R\ot \delta_{H}) $
\item[ ]$= (S\ot c_{M,R}\ot c_{H,N})\co (S\ot M\ot c_{H,R}\ot N)\co (((S\ot \rho_{M})\co h_{M}^{-1})\ot
h_{N}^{-1})\co (S\ot c_{H,R}\ot H)\co (S\ot R\ot \delta_{H})  $
\item[ ]$= (S\ot c_{M,R}\ot c_{H,N})\co (S\ot M\ot c_{H,R}\ot N)\co
((( h_{M}^{-1}\ot H)\co (S\ot \delta_{H}))\ot h_{N}^{-1})\co (S\ot
c_{H,R}\ot H)\co (S\ot R\ot \delta_{H})  $
\item[ ]$= (g_{M\ot N}^{\prime}\ot H)\co (S\ot R\ot ((H\ot \delta_{H})\co \delta_{H}))  $
\item[ ]$=(S\ot R\ot \rho_{M\ot N}^{2})\co g_{M\ot N}^{\prime}\co (S\ot R\ot \delta_{H})   $
\end{itemize}
As a consequence, there exists a unique morphism $h_{M\bullet
N}^{\prime}: T\ot H\rightarrow T\ot M\bullet N $ such that
\begin{equation}
\label{product-geo-nor-bas-2} (T\ot i_{M\bullet N})\co h_{M\bullet
N}^{\prime}=g_{M\ot N}^{\prime}\co (T\ot \delta_{H}).
\end{equation}

Thus, by (\ref{product-geo-nor-bas-1}) and
(\ref{product-geo-nor-bas-2})
$$h_{M\bullet
N}\co h_{M\bullet N}^{\prime}= (T\ot ((\varepsilon_{H}\ot H)\co
\delta_{H}))\co h_{M\bullet N}\co h_{M\bullet N}^{\prime}=(T\ot
\varepsilon_{H}\ot H)\co g_{M\ot N}\co (T\ot i_{M\bullet N})\co
h_{M\bullet N}^{\prime}$$ $$=(T\ot \varepsilon_{H}\ot H)\co g_{M\ot
N}\co g_{M\ot N}^{\prime}\co (T\ot \delta_{H})=id_{T\ot H}$$ and
$$
(T\ot i_{M\bullet N})\co h_{M\bullet N}^{\prime}\co h_{M\bullet N}
=g_{M\ot N}^{\prime}\co (T\ot \delta_{H})\co h_{M\bullet N}= g_{M\ot
N}^{\prime}\co g_{M\ot N}\co (T\ot i_{M\bullet N})=T\ot i_{M\bullet
N}$$ and then $h_{M\bullet N}$ is an isomorphism with inverse
$h_{M\bullet N}^{-1}=h_{M\bullet N}^{\prime}$.

The morphism $h_{M\bullet N}$ is a morphism of right $H$-comodules
because
\begin{itemize}
\item[ ]$\hspace{0.38cm} (h_{M\bullet N}\ot H)\co (T\ot \rho_{M\bullet N}) $
\item[ ]$= (T\ot ((H\ot \varepsilon_{H})\co \delta_{H})\ot H)\co
(h_{M\bullet N}\ot H)\co (T\ot \rho_{M\bullet N})   $
\item[ ]$= (T\ot H\ot \varepsilon_{H}\ot H)\co (g_{M\ot N}\ot H)\co
(T\ot ((i_{M\bullet N}\ot H)\co \rho_{M\bullet N}))  $
\item[ ]$=(T\ot H\ot \varepsilon_{H}\ot H)\co (g_{M\ot N}\ot H)\co
(T\ot ((M\ot \rho_{N})\co i_{M\bullet N}))   $
\item[ ]$= (T\ot H\ot ((H\ot \varepsilon_{H})\co \delta_{H}))\co g_{M\ot N}\co (T\ot i_{M\bullet N})  $
\item[ ]$=(T\ot \delta_{H})\co h_{M\bullet N}   $
\end{itemize}
where the first equality follows by the counit property, the second
and the last ones by (\ref{product-geo-nor-bas-1}), the third one
the properties of $\rho_{M\bullet N}$ and the fourth one by the
comodule condition for $h_{N}$.

Finally, we will  prove that $h_{M\bullet N}$ is almost lineal.
Indeed:
\begin{itemize}
\item[ ]$\hspace{0.38cm} (\mu_{T}\ot H)\co (T\ot (h_{M\bullet N}\co (\eta_{T}\ot M\bullet N))) $
\item[ ]$=(\mu_{T}\ot ((\varepsilon_{H}\ot H)\co \delta_{H}))\co (T\ot (h_{M\bullet N}\co
(\eta_{T}\ot M\bullet N)))    $
\item[ ]$= (\mu_{S\ot R}\ot \varepsilon_{H}\ot H)\co (S\ot R\ot (g_{M\ot N}\co (\eta_{S}\ot \eta_{R}\ot
i_{M\bullet N})))  $
\item[ ]$=  (S\ot \varepsilon_{H}\ot R\ot H)\co (((\mu_{S}\ot H)\co (S\ot (h_{M}\co (\eta_{S}\ot M))))\ot
((\mu_{R}\ot H)\co (R\ot (h_{N}\co (\eta_{R}\ot N)))))$
\item[ ]$\hspace{0.38cm} \co (S\ot c_{R,M}\ot N)\co (S\ot
R\ot i_{M\bullet N}) $
\item[ ]$=(T\ot \varepsilon_{H}\ot  H)\co g_{M\ot N}\co (T\ot i_{M\bullet N})  $
\item[ ]$= (T\ot ((\varepsilon_{H}\ot H)\co \delta_{H}))\co h_{M\bullet N} $
\item[ ]$=h_{M\bullet N}   $
\end{itemize}
In the last equalities, the first  and the sixth ones follow by the
properties of the counit, the second and the fifth ones by
(\ref{product-geo-nor-bas-1}), the third one is a consequence of the naturality of
$c$ and the fourth one relies on the almost lineal condition for $h_{M}$
and $h_{N}$.
\end{proof}

As a direct consequence of this proposition we have the
following theorem.

\begin{teo}
\label{main-2} Let $H$ be a cocommutative faithfully flat Hopf
quasigroup. If we denote by $P_{gnb}(K,H)$ the category whose
objects are the invertible right $H$-comodules with geometric normal
basis and whose morphisms are the morphisms of right $H$-comodules
between them, $P_{gnb}(K,H)$ with the product defined in the previous
proposition is a symmetric monoidal category where the unit object
is $\mathbf{H}$ and the symmetry isomorphisms, the left, right an
associative constraints are defined as in Proposition
\ref{monoidal-magmas}. Moreover, the set of isomorphism classes in
$P_{gnb}(K,H)$ is a monoid that we will denote by $Pic_{gnb}(K,H)$.
\end{teo}

\begin{rem}
\label{fin-1} {\rm There is a  monoid morphism $\omega:
Gal_{{\mathcal C}}(H)\rightarrow P_{gnb}(K,H)$ defined by
$\omega([{\Bbb A}])=[\mathbf{A}]$. If $\omega([{\Bbb
A}])=[\mathbf{H}]$ we have that $A\cong H$ as right $H$-comodules.
Then $[{\Bbb A}]\in N_{{\mathcal C}}(H)$. Also, if $[{\Bbb A}]\in
Gal^s_{{\mathcal C}}(H)$ and $\omega([{\Bbb A}])=[\mathbf{H}]$,
$[{\Bbb A}]\in N^s_{{\mathcal C}}(H)$.

}
\end{rem}

\section*{Acknowledgements}
The authors  were supported by  Ministerio de Econom\'{\i}a y Competitividad (Spain) and by Feder founds. Project MTM2013-43687-P: Homolog\'{\i}a, homotop\'{\i}a e invariantes categ\'oricos en grupos y \'algebras no asociativas.

\end{document}